\documentclass[12pt]{amsart}
\usepackage{hyperref}
\usepackage{amsmath, amsthm, amstext}
\usepackage{amssymb, array, amsfonts}
\usepackage[english]{babel}
\usepackage{float}
\usepackage{bm}
\usepackage{graphics}
\usepackage{graphicx}
   \usepackage{enumerate}
\numberwithin{equation}{section}

\usepackage[margin=1.2in]{geometry}

\def \ER{\mathcal{ER}}
\def \g{\mathcal{G}}
\def \b{\mathcal{B}}

\def \dc{\mathrm{DC}}

\renewcommand{\l}{\left}
\renewcommand{\r}{\right}

\def \const{\mathrm{const}}

\def \E{\mathcal{E}}

\def \M2{\mathrm{M}_2}
\def \R{\mathbb{R}}
\def \Seps{S_{\varepsilon}}
\def \Z{\mathbb{Z}}
\def \T{\mathbb{T}}
\def \A{\mathcal{A}}
\def \sl2r{\mathrm{SL}(2,\R)}

\newcommand{\beq}{\begin{equation}}
\newcommand{\eeq}{\end{equation}}
\newcommand{\one}{\mathbf{1}}

\def\ran{\operatorname{ran}}

\def\diam{\operatorname{diam}}
\def\im{\operatorname{Im}}

\newcommand{\eqdef}{\stackrel{\rm def}{=\kern-3.6pt=}}

\theoremstyle{plain}
\newtheorem{theorem}{\bf Theorem}[section]
\newtheorem{lemma}[theorem]{\bf Lemma}

\newtheorem{prop}[theorem]{\bf Proposition}
\newtheorem{cor}[theorem]{\bf Corollary}

\theoremstyle{definition}
\newtheorem{defi}[theorem]{\bf Definition}

\theoremstyle{remark}
\newtheorem{remark}[theorem]{\bf Remark}

\renewcommand{\le}{\leqslant}
\renewcommand{\ge}{\geqslant}

\newcommand{\dist}{\mathop{\mathrm{dist}}\nolimits}

\renewcommand{\qed}{\vrule height7pt width5pt depth0pt}
\title[Interacting quasiperiodic particles]{Anderson localization for two interacting quasiperiodic particles}
\author[J. Bourgain]{Jean Bourgain}
\address{School of Mathematics,
	Institute for Advanced Study,
	Princeton, NJ, 08540,
	United States of America}
\email{bourgain@math.ias.edu}

\author[I. Kachkovskiy]{Ilya Kachkovskiy}
\address{Department of Mathematics,
	Michigan State University,
	East Lansing, MI, 48824,
	United States of America}
\email{ikachkov@msu.edu}

\date{}
\begin{document}
\maketitle
\section{Introduction}
In this paper, we study Anderson localization for the following family of Schr\"odinger operators on $\ell^2(\mathbb Z^2)$:
\beq
\label{h_def}
H(\theta_1,\theta_2)=\Delta+\lambda(v(n_1\omega+\theta_1)+v(n_2\omega+\theta_2))+U(n_1,n_2).
\eeq
Here $\Delta$ is the discrete Laplacian, $v\in C^{\omega}(\T)$ is a real analytic function (identified with a $1$-periodic analytic function on $\mathbb R$), $\omega$ is an irrational number, and $\theta=(\theta_1,\theta_2)\in \T^2$ is the quasiperiodic phase. The case $U=0$ corresponds to a direct sum of two 1D quasiperiodic operators, which can be treated as a system of two non-interacting particles on $\mathbb Z$. The function $U$ in \eqref{h_def} is a (deterministic) interaction potential. A typical example of $U$ would be a finite-range interaction, that is, $U(n_1,n_2)=f(n_1-n_2)$, where $f$ is a function on $\Z$ supported on a finite set. In general, we consider general potentials $U$ of low complexity (see Section 2.1). In particular, they will always take finitely many values.

Our main motivation for studying the models \eqref{h_def} comes from the localization results for interacting particles in random environments, see \cite{AW,KL,CHU}. In one way or another, random analogues of operators \eqref{h_def} demonstrate Anderson localization at large disorder $\lambda$. Another phenomenon is the following decoupling result: suppose the single particle is localized for some $\lambda$. Then, for sufficiently small finite range interaction $U$, where the smallness depends on $\lambda$, the multi-particle system will also be localized. Unfortunately, despite some progress obtained in the present paper, a similar question in the quasiperiodic case remains largely open, and our current results are obtained using 2D methods.

The second source of motivation, which seems to be closer to applications, comes from the numerical work \cite{FS}, see also references therein. They authors considered the almost Mathieu case $v(\theta)=\cos (2\pi \theta)$ and conjectured that there are some regimes where the addition of the interaction potential can generate some delocalized states (FIKS, that is, freed by interaction kinetic states), based on numerical evidence. In fact, their work goes as far as proposing real-life experiments with cold atoms, that can possibly confirm this prediction. 

Assuming that $U$ has low complexity in the sense of Section 2.1, localization results of the present paper depend on whether or not the quasiperiodic potential $v$ has cosine-type symmetries. Generic analytic potentials $v$ have no symmetries, and in this case we are able to obtain a perturbative localization result for large $\lambda$, assuming that a small positive measure set set of frequencies is removed, see Theorem \ref{main_asym}. This includes the regime of strong interaction, as long as its strength is $O(|\lambda|)$, see \eqref{eq_u_notlarge}. One can consider this result as an evidence of the fact that, in the asymmetric case, interaction of low complexity cannot break down the localization (however, the initial regime may need stronger disorder than just for single-particle localization).

In the case of potentials with symmetries, such as $v(\theta)=\cos(2\pi \theta)$, we still obtain localization at large disorder, however, we have to remove a part of spectrum of size $o(|\lambda|)$ around the (finitely many) values of $U$ from consideration, see Theorem \ref{main_sym}. Thus, possible FIKS are restricted to neighborhoods of finitely many energies, which comprise a relatively small part of the spectrum of the operator \eqref{h_def}.

While we do not obtain any delocalization results in the present paper, the almost Mathieu version of \eqref{h_def} is studied in \cite{BJK} (currently in preparation), in the regime of strong Hubbard-type interaction: that is, 
$$
v(\theta)=2\cos(2\pi\theta),\quad U(n_1,n_2)=u\delta_{n_1,n_2}, \quad  \lambda\text{ is fixed},\quad u\gg \lambda.
$$
In this case, the operator \eqref{h_def} has some spectrum in the region $[u-4-4\lambda,u+4+4\lambda]$, which is separated from the ``bulk'' spectrum. One can show that, for fixed phase difference $\theta_1-\theta_2$ and large $\lambda$, the spectrum in the interval $[u-4-4\lambda,u+4+4\lambda]$ is purely point, for a full measure set of frequencies $\omega$. On the other hand, for fixed $u\gg\lambda$ and some subsequent choice of $\theta_1-\theta_2\approx \pi$ (depending on $u$), the operator \eqref{h_def} has some non-trivial absolutely continuous spectrum in that interval. Thus, in the symmetric case, delocalization can indeed happen in some of the regions excluded in Theorem \ref{main_sym} away from the spectrum of the non-interacting operator (the latter is essential for the analysis in \cite{BJK}). In the case where excluded energies belong to the non-interacting spectrum, the question of localization/delocalization remains largely open.

\section{Summary of main results}
We start from the description of the classes of potentials $v$ and $U$ that can be considered.
\subsection{Low complexity interaction potentials}
Let $T_{(n_1,n_2)}$ be the translation operator on $\ell^2(\Z^2)$:
$$
(T_{n_1,n_2}\psi)(m_1,m_2)=\psi(m_1-n_1,m_2-n_2).
$$
\begin{defi}
Let $U\colon \Z^2\to \R$. For each $(n_1,n_2)\in \Z^2$, consider the function $T_{n_1,n_2}U$ restricted to $[0,N-1]^2$. 
We say that $U$ has {\it low complexity} if the number of different functions among such restricted translations of $U$ admits a power bound in $N$:
\beq
\label{eq_lowcomp_def}
\#\l\{\one_{[0,N-1]^2}\cdot \l(T_{n_1,n_2}U\r): n_1,n_2\in \Z\r\}\le N^{C_{\mathrm{int}}},\quad  (n_1,n_2)\in \Z^2,\quad N\ge 2,
\eeq
for some constant $C_{\mathrm{int}}>0$.
\end{defi}
In other words, $U$ has low complexity if one can get at most $N^{C_{\mathrm{int}}}$ possible configurations by restricting $U$ to a box of size $N$. This condition implies that $U$ can only take finitely many values:
\beq
U(n_1,n_2)\in \{U_1,\ldots,U_{N_{\mathrm{int}}}\},
\eeq
where, say, $N_{\mathrm{int}}\le 4^{C_{\mathrm{int}}}$.
The class of possible $U$ contains all periodic potentials on $\Z^2$, all finite range translationally invariant interaction potentials, and, in fact, some interesting examples such as Fibonacci-type potentials (see \cite{DGY} and references therein). 

We will also assume that $U$ is not very large compared to the disorder:
\beq
\label{eq_u_notlarge}
|U(n_1,n_2)|\le m_{\mathrm{int}}|\lambda|,\quad \text{for some}\,\,m_{\mathrm{int}}>0.
\eeq
The possibility of considering background potentials of low complexity can already be traced back to the methods of \cite{BGS}, although it was not stated there. While some low complexity potentials (such as Fibonacci hamiltonians) are known to cause singular continuous spectra by themselves even at small coupling, our results show that these effects are dominated by localization caused by quasiperiodic potentials at large disorder.

\subsection{Symmetric and asymmetric single-particle potentials}
The results also depend on whether or not the potential $v$ has certain symmetries. In the latter, we will always assume that $v\in C^{\omega}(\T;\R)$ and will identify $C^{\omega}(\T,\R)$ with the space of all $1$-periodic real analytic functions on $\R$. In addition, we will assume, without loss of generality, that $\int_{\T}v(\theta)\,d\theta=0$ and that $1$ is the smallest period of $v$. We call $v$ {\it symmetric} if at least one of the following conditions holds:
\begin{enumerate}
	\item $v(\theta_{\mathrm{sym}}+\theta)=-v(\theta_{\mathrm{sym}}-\theta)$ for some $\theta_{\mathrm{sym}}\in \T$ and all $\theta$.
	\item $v(\theta+1/2)=-v(\theta)$ for all $\theta$.
\end{enumerate}
This symmetries will be called Type I and Type II, respectively. If none of them holds, we will call $v$ {\it asymmetric}. In case of Type I symmetry, by shifting the function, {\it we can assume without loss of generality that $\theta_{\mathrm{sym}}=0$}. For example, the (shifted) almost Mathieu potential $\cos(2\pi \theta-\pi/2)=\sin 2\pi \theta$ satisfies both symmetries, the function $\sin2\pi \theta+\sin 4\pi \theta $ has only Type I, and $\cos 2\pi \theta+\sin 6\pi \theta $ has only Type II symmetry.

\subsection{Main results}
The following two theorems are main results of the present paper.
\begin{theorem}
\label{main_asym}
Suppose that $v$ is asymmetric. For any $\varepsilon_{\mathrm{freq}}>0$ there exists $\lambda_0=\lambda_0(v,\varepsilon_{\mathrm{freq}},m_{\mathrm{int}},C_{\mathrm{int}})$ such that the following is true: for every $\lambda\ge\lambda_0$, every $(\theta_1,\theta_2)\in \T^2$, and every interaction potential $U$ satisfying the assumptions of Section $2.1$, there is a set 
$$
\Omega(U,\lambda,\theta_1,\theta_2)\subset \T,\quad  |\T\setminus \Omega(U,\lambda,\theta_1,\theta_2)|<\varepsilon_{\mathrm{freq}},
$$ 
with the operator $H(\theta_1,\theta_2)$ satisfying complete Anderson localization for all $\omega\in \Omega(U,\lambda,\theta_1,\theta_2)$ and for all possible translations of $U$.
\end{theorem}
\begin{theorem}
\label{main_sym}
Fix a background potential $U$ of low complexity. Suppose $v$ is symmetric and admits a bounded analytic extension into $|\im z|\le 20$. For any $\varepsilon_{\mathrm{freq}}>0$ there exist $\lambda_0=\lambda_0(v,\varepsilon_{\mathrm{freq}},m_{\mathrm{int}},C_{\mathrm{int}})$ and $\mu=\mu(v,\varepsilon_{\mathrm{freq}},m_{\mathrm{int}},C_{\mathrm{int}})>0$ such that, for every $\lambda\ge\lambda_0$, every $(\theta_1,\theta_2)\in \T^2$, and every interaction potential $U$ satisfying the assumptions of Section $2.1$, there exists a set 
$$
\Omega(U,\lambda,\theta_1,\theta_2)\subset \T,\quad |\T\setminus \Omega(U,\lambda,\theta_1,\theta_2)|<\varepsilon,
$$ with the operator $H(\theta_1,\theta_2)$ satisfying Anderson localization in the region of energies $E$
\beq
\label{energyrange}
|E-U_j|\ge \frac{\lambda}{e^{(\log\lambda)^\mu}},\quad j=1,\ldots,N_{\mathrm{int}},
\eeq
for all $\omega\in \Omega(U,\lambda,\theta_1,\theta_2)$ and for all possible translations of $U$.
\end{theorem}
\begin{remark}
\label{main_remarks}	
\begin{enumerate}
	\item The part of the spectrum removed by \eqref{energyrange} is contained in $N_{\mathrm{int}}$ intervals of size $o(|\lambda|)$.
	\item The condition on separability of $v$ is irrelevant in Theorem \ref{main_asym}.  The potential $v(\theta_1)+v(\theta_2)$ can be replaced by an analytic function $w(\cdot,\cdot)\in C^{\omega}(\T^2)$ of two variables that is not constant on any straight line segment, with the same proof. 
	\item  An analogue of Theorem \ref{main_sym} can also be obtained for non-separable case, assuming $|E-(\lambda w_i+U_j)|\ge \lambda\varepsilon$ for every $w_i$ such that $w(\theta_1,\theta_2)\equiv w_i$ on some straight line segment, and every value $U_j$ of $U$, also with the same proof.
	\item As discussed above, the inclusion of a background potential $U$ of low complexity could have been done already in \cite{BGS}, as well as in the other papers that establish perturbative results by semi-algebraic techniques and do not involve Lyapunov exponents/cocycles (for example, in \cite{BJ_band}).
	\item Suppose that $v$ satisfies Type II symmetry, $U=0$, $\theta_1=\theta_2=1/4$. Then one can easily check that $\psi(n_1,n_2)=(-1)^{n_1}\delta_{n_1 n_2}$ solves the eigenvalue equation $H(1/4,1/4)\psi=0$. While this does not contradict purely point spectrum, all known proofs of Anderson localization show that any solution of the eigenvalue equation decays exponentially, which does not allow the existence of states like $\psi$. This example suggests that some stronger versions of localization can break down at zero energy, but only in symmetric cases (because otherwise Theorem \ref{main_asym} holds). Possible scenarios of delocalization at zero energy in different models are described in \cite{JB,ES}.
    \item The condition in Theorem \ref{main_sym} of analyticity of $v$ in the strip of size 20 is technical and can possibly be removed with some extra work.

	\item The operator family $H(\theta_1,\theta_2)$ is not ergodic because $U$ is not assumed to have any translational invariance. However, one can prove localization simultaneously for all translations of $U$.
	\item Our results are perturbative, in the sense that one always has to remove a positive measure set of frequencies. However, our requirements on the frequency are more explicit. The bound on $\lambda_0$ can be expressed, in principle, through the Diophantine constant $C_{\mathrm{dio}}$ of $\omega$ (see \eqref{eq_diophantine}). The parameter $\varepsilon_{\mathrm{freq}}$ in Theorems \ref{main_asym} and \ref{main_sym} is, essentially, the measure of frequencies for which \eqref{eq_diophantine} does not hold with this $C_{\mathrm{dio}}$. Afterwards, as usually happens in localization proofs, one has to remove an extra set of frequencies of measure zero, for which we do not have any arithmetic description, and which depends on $\theta_1$, $\theta_2$, $\lambda$, and other parameters.
	\item Theorems \ref{main_asym} and \ref{main_sym} are formulated for the case of a single phase $(\theta_1,\theta_2)$. However, one can extend them for a full measure set of phases, see 
Remark \ref{full_measure}. We do not believe this argument is new, however, in the case of perturbative results, it has not been explicitly stated in the literature.
\item In the case when $\|U\|_{\infty}$ is bounded by a constant independent of $\lambda$, the removed energy intervals in Theorem \ref{main_sym} become one neighborhood of zero energy of the same size.
\end{enumerate}
\end{remark}
\subsection{Structure of the paper} The main lines of argument are parallel to the only known techniques of establishing multi-dimensional Anderson localization \cite{BGS,B4}. The proof consists of two main steps: a large deviation theorem for Green's function at fixed energy (Sections 3 -- 6), and elimination of energy (Sections 7 -- 9).

In the large deviation result, we need to obtain stronger inductive assumptions to carry over from a scale to the next scale. That is, we require that the large deviation set has small sections by lines in all directions, rather than in the coordinate directions only. In Section 3, we establish these bounds at the initial scale. In the case of no symmetries, the initial scale bounds at all energies follow from Proposition \ref{initial_nosym} and are essentially known. In the case when symmetries are present, similar argument would immediately work for $|E-U_j|\ge \varepsilon|\lambda|$. To get better initial scale bounds in the wider region \eqref{energyrange}, one has to apply 1D large deviations more carefully, as done in Section 3. Section 4 contains preliminaries from theory of semi-algebraic sets (some proofs are provided in the Appendix). In order to pass to the next scale, one needs an arithmetic bound on the number of bad boxes (cf. \cite[Section 3]{BGS}). In our proof, a stronger inductive assumption is carried over to the next scale, but there is less freedom in removal of frequencies, as the frequency vector is always of the form $(\omega,\omega)$. Still, a relatively simple argument in Section 5 shows that one can get a sub-linear bound on the number of bad sites assuming a Diophantine condition on $\omega$. In Section 6, we provide the inductive argument for obtaining Green's function estimates at fixed energy. After the preparations from previous sections, the proof goes along the lines of \cite{BGS}. In particular, our Proposition \ref{bgs_44} is an analogue of \cite[Lemma 4.4]{BGS}. For the convenience of the reader, we include the proof based on Cartan's lemma in the Appendix. Our Corollary \ref{green_main} is the main result of that section and is similar to \cite[Proposition 4.6]{BGS}; however, the control of the constants is somewhat more delicate, and we provide the argument in the main text.

The elimination of energy requires more work, as the steep planes argument used in \cite{BGS}, \cite{B4} significantly relies on the fact that the dimension of the lattice is the same as the dimension of the set of frequency vectors, which is not the case in our situation. However, our large deviation theorem is stronger, and we take advantage of that. In Section 7, we use more elaborate semi-algebraic arguments and bounds on Kakeya maximal functions, in order to estimate the number of directions in which the large deviation set can contain a long line segment, see Lemmas \ref{elimination_intervals}, \ref{kakeya}. In Section 8, we show that double resonances can be avoided by removing a small frequency set, if one combines the results of Section 7 with some careful choice of scales. Section 9 completes the proof of localization, which at this point becomes fairly standard. In Remark \ref{full_measure}, we also explain how to obtain results for sets of phases of full measure. While this argument is not new, it has not been explicitly stated in previous works.

\section{Symmetries of the potential and the initial scale}
Everywhere in this section, we will assume
\begin{enumerate}
	\item $v\in C^{\omega}(\T;\R)$.
	\item $1$ is the smallest period of $v$ (in particular, $v\neq\mathrm{const}$).
	\item $\int_{\T}v(\theta)\,d\theta=0$.
	\item $|v(\theta)|\le 1$ for all $\theta\in \T$.
\end{enumerate}
The results will heavily rely on the structure of level sets of the function $v(\theta_1)+v(\theta_2)$. We will have to avoid the situation when these sets contain straight line segments.
\begin{lemma}
\label{symmetries}
The level set $v(\theta_1)+v(\theta_2)=E$ contains a straight line segment if and only if $E=0$ and $v$ is symmetric. In the case of Type I symmetry $v(\theta_{\mathrm{sym}}+\theta)=-v(\theta_{\mathrm{sym}}-\theta)$ this segment has the equation $\theta_{\mathrm{sym}}-\theta_1=\theta_{\mathrm{sym}}+\theta_2$, and in case of type II symmetry $v(\theta+1/2)=-v(\theta)$ this segment is $\theta_1=\frac12+\theta_2$ $($both equalities are modulo $\Z)$.
\end{lemma}
\begin{proof}
	Suppose that $v(\theta)+v(a\theta+b)\equiv E$. Since $1$ is the smallest period of $v$, we have $a=\pm 1$, and hence, from comparing the mean values, we have $E=0$.
Suppose that $a=1$. Then $v(\theta)=-v(\theta+b)=v(\theta+2b)$, and hence $b=1/2$. If $a=-1$, then $v(\theta)=-v(b-\theta)$. Then the function $v_1(\theta)=v(b/2+\theta)$ is odd, and has same symmetry for $a=1$ if and only if $v$ has it.
\end{proof}
In the rest of this section, we obtain Green's function estimates at the initial scale for $H$. We will need to use some large deviation theorems for analytic functions. The case of asymmetric $v$ is essentially known.
\begin{prop}
\label{initial_nosym}
Suppose $w\in C^{\omega}(\T^2)$ is non-constant on any line segment in $\T^2$. Then there are positive constants $C,\alpha$ depending only on $v$ such that, for any unit line segment $L\subset \mathbb R^2$, any $E\in \mathbb R$, and any $\delta>0$, we have
\beq
\label{nosym_bound}
|\{(\theta_1,\theta_2)\in L\colon |w(\theta_1,\theta_2)-E|\le \delta\}|_1\le C\delta^{\alpha},
\eeq
where $|\cdot|_1$ denotes the one-dimensional Lebesgue measure. The constant $C$ can be chosen uniformly in $E$ on any compact interval.
\end{prop}
\begin{proof}
Define an analytic function $f\in C^{\omega}((-1,1)^3)$ by
$$
f(\theta_1,\theta_2,\eta)=w(\theta_1 \cos 2\pi \eta+\theta_2 \sin 2\pi\eta,-\theta_1 \sin 2\pi \eta+\theta_2 \cos 2\pi \eta).
$$
The function $f(\cdot, \theta_2,\eta)$ is non-constant in $\theta_1$ at any fixed $\theta_2,\eta$. Moreover, by the choice of $\theta_2$ and $\eta$, one can parametrize any line segment with $\theta_1$ being the natural length parameter. Then, one can refer to the beginning of Section 4 of \cite{BGS} and \cite[Lemma 11.4]{GS}.
\end{proof}
\begin{remark}
\label{nosym_bound_remark}	
If $v$ is asymmetric, then $w(\theta_1,\theta_2)=v(\theta_1)+v(\theta_2)$ satisfies the assumptions of Proposition \ref{nosym_bound}.
\end{remark}

We now consider the symmetric separable case. Any line segment in $L\subset \T^2$ can be parametrized either by $\theta_1$ or $\theta_2$. Since $v$ is $1$-periodic, the restriction of $v(\theta_1)+v(\theta_2)$ onto $L$ can be completely described by the function
\beq
\label{vsym_theta}
\theta\mapsto v(\theta)+v(a\theta+b),\quad -1\le a\le 1, \,\,0\le b\le 1.
\eeq
Let 
$$
g(v,a,b)=\max\limits_{-1/2\le\theta\le 1/2}|v'(\theta)+av'(a\theta+b)|.
$$
\begin{lemma}
\label{grad_lemma}
Suppose that $v$ is symmetric. Let also $-1\le a\le 1$, $0\le b<1$. Then, for some $C(v)>0$, we have
\begin{enumerate}
\item If $v$ has only Type I symmetry with $\theta_{\mathrm{sym}}=0$, then
\beq
\label{eq_lowerg_1}
	g(v,a,b)\ge C(v)(|a+1|+b(1-b))
\eeq
\item If $v$ has only Type II symmetry, then
\beq
\label{eq_lowerg_2}
g(v,a,b)\ge C(v)(|a-1|+|b-1/2|).
\eeq
\item If $v$ has both Type I (with $\theta_{\mathrm{sym}}=0$) and Type II symmetries, then
\beq
\label{eq_lowerg_3}
g(v,a,b)\ge C(v)(|a^2-1|+b(1-b)|b-1/2|).
\eeq
\end{enumerate}
\end{lemma}
\begin{proof}
The inequality 
\beq
\label{eq_lower_a}
g(v,a,b)\ge C_1\min\{|a-1|,|a+1|\}
\eeq
is clearly satisfied in all three cases with $C_1=C_1(v)=\max_{\theta\in \T}|v'(\theta)|$ (since the first term in the derivative of \eqref{vsym_theta} attains its maximum for some $\theta\in [0,1]$). Let us also note that, in all three cases, $g(v,a,b)$ is Lipschitz in $a$ uniformly in $b$:
\beq
\label{eq_g_lipschitz}
|g(v,a_1,b)-g(v,a_2,b)|\le C_{\mathrm{Lip}}|a_1-a_2|,\quad C_{\mathrm{Lip}}=C_{\mathrm{Lip}}(v)>0.
\eeq
We now address each of the three cases separately.

{\noindent\it Case 1.} Due to continuity of $g$ and absence of Type II symmetry, for any $\varepsilon>0$, we have 
$$
g(v,a,b)\ge C(\varepsilon)>0,\quad \text{for}\quad -1+\varepsilon\le a\le 1,\,\, 0\le b\le 1.
$$ 
Hence, \eqref{eq_lower_a} can be replaced by $g(v,a,b)\ge C_1 |a+1|$, which implies \eqref{eq_lowerg_1} for $C_{\mathrm{Lip}}|a+1|\ge \frac12 b(1-b)$. Suppose now that $C_{\mathrm{Lip}} |a+1|<\frac12 b(1-b)$. Then, \eqref{eq_g_lipschitz} implies that it would be sufficient to establish \eqref{eq_lowerg_1} for $a=-1$, that is, to estimate $g(v,-1,b)$ from below. Let 
$$
v(\theta)=\sum\limits_{n\in \mathbb Z}c_n e^{2\pi i n \theta}.
$$
Then
$$
g(v,-1,b)^2\ge\int_0^1|v'(\theta)-v'(b-\theta)|^2\,d\theta=\sum\limits_{n\in \mathbb Z}4n^2|c_n|^2\sin^2(\pi nb).
$$
Since 1 is the smallest period of $v$, there is a finite index set $I\subset \Z$ with $\gcd(I)=1$ and $c_n\neq 0$ for $n\in I$, which implies
$$
g(v,-1,b)^2\ge \sum\limits_{n\in I}4n^2|c_n|^2\sin^2(\pi nb)\ge C b^2(b-1)^2,
$$
since the last expression can only vanish for $b=0$ or $b=1$, and in both cases admits a quadratic lower bound.

{\noindent\it Case 2. }Similarly to Case 1, \eqref{eq_lower_a} can be replaced by $g(v,a,b)\ge C_1 |a-1|$, which implies \eqref{eq_lowerg_2} for $C_{\mathrm{Lip}}|a+1|\ge \frac12 |b-1/2|$. Hence, we can assume that the opposite inequality holds, which allows to consider $a=1$ and obtain a similar Fourier estimate:
$$
g(v,1,b)^2\ge \sum\limits_{n\in \mathbb Z}4n^2|c_n|^2\cos^2(\pi nb)\ge \sum\limits_{n\in I}4n^2|c_n|^2\cos^2(\pi nb),
$$
where $I\subset \Z$ is an index set with $\gcd(I)=1$ and $c_n\neq 0$ for $n\in I$. Type II symmetry implies that all such $n$ must be odd, and hence the right hand side can only vanish for $b=1/2$, with any of the terms providing a lower bound $C(b-1/2)^2$.

{\noindent\it Case 3. }Similarly to Case 1 and Case 2, \eqref{eq_lower_a} immediately implies \eqref{eq_lowerg_3} for $C_{\mathrm{Lip}}|a^2-1|\ge \frac14 b(1-b)|b-1/2|$.
In case of the opposite inequality, one can replace $a$ by $1$ or $-1$ using \eqref{eq_g_lipschitz} and then apply the same Fourier lower bound from Case 2 or Case 1, respectively.
\end{proof}
One can also obtain upper bounds, which immediately follow from Lipschitz continuity of $v$, and combine both results into the following
\begin{cor}
\label{cor_upper}
Under the assumptions of Lemma $\ref{grad_lemma}$, we have the following two-sided bounds on $-\sqrt{2}\le \theta\le \sqrt{2}$ with $C_-(v),C_+(v)>0$
\begin{enumerate}
\item If $v$ has only Type I symmetry with $\theta_{\mathrm{sym}}=0$, then
$$
C_-(v)(|a+1|+b(1-b))\le g(v,a,b)\le C_+(v)(|a+1|+|b(b-1)|).
$$
\item If $v$ has only Type II symmetry, then
$$
C_-(v)(|a-1|+|b-1/2|)\le g(v,a,b)\le C_+(v)(|a-1|+|b-1/2|).
$$
\item If $v$ has both Type I (with $\theta_{\mathrm{sym}}=0$) and Type II symmetries, then
$$
C_-(v)(|a^2-1|+b(1-b)|b-1/2|)\le g(v,a,b)\le C_+(v)(|a^2-1|+b(1-b)|b-1/2|).
$$
\end{enumerate}
Moreover, each upper bound also holds for $|v(\theta)+v(a\theta+b)|$ uniformly in $\theta\in [0,1]$.
\end{cor}
\begin{remark}
\label{nosym_remark}
In a separable asymmetric case, similar arguments imply that $g(v,a,b)\ge \varepsilon(v)>0$ uniformly in $a,b$.
\end{remark}
\begin{remark}
In case of Type I symmetry with $\theta_{\mathrm{sym}}\neq 0$, the second term in Cases 1 and 3 will be different due to the shift of $b$. As mentioned earlier, we will always assume $\theta_{\mathrm{sym}}=0$.
\end{remark}
\begin{prop}
\label{harnack}
Let $f$ be an analytic function in the disk $|z|\le 2e$, $|f(z)|\le M$ for $|z|\le 2e$, and $f(0)=1$. Let $D=\{z\colon|z|\le 1,\,|f(z)|\le \lambda\}$. Then $D$ can be covered by a union of disks of total diameter bounded by $C \exp\{\frac{\log\lambda}{\log M}\}$ $($where $C$ is an absolute constant$)$. In particular,
$$
|D\cap[-1,1]|_1\le C \exp\l\{2\frac{\log\lambda}{\log M}\r\}.
$$
\end{prop}
As earlier, $|\cdot|_1$ denotes the 1D Lebesgue measure. For the proof, see Theorem 4 in Section 11.3 of \cite{L}.
\begin{theorem}
\label{initial_ldt}
Suppose $v$ is symmetric and extends to a bounded analytic function in the strip $|\im z|\le 20$. Then, for any line segment $L\subset \mathbb R^2$ of unit length, and any $E\in [-3,3]\setminus\{0\}$, $\delta>0$, we have
\beq
\label{sym_bound}
|\{(\theta_1,\theta_2)\in L\colon |v(\theta_1)+v(\theta_2)-E|\le \delta\}|_1\le c_1(v)\delta^{-\frac{c_2(v)}{\log |c_3(v)E|}},
\eeq
where $|\cdot|_1$ denotes the one-dimensional Lebesgue measure and $c_1(v),c_2(v)>0$, $0<c_3(v)\le 1/6$ depend only on $v$.
\end{theorem}
\begin{proof}
We will assume that $L$ is described by $a,b$ as in \eqref{vsym_theta}, and the points of $L$ are parametrized by $\theta$ (which is either $\theta_1$ or $\theta_2$). 
Without loss of generality, we can also assume that $\delta\le |E|/2$, otherwise the bound can be obtained by choosing a sufficiently large $c_1(v)$. Finally, one only needs to consider $a$, $b$ satisfying
\beq
\label{eq_condition_ab}
g(v,a,b)\ge \frac{C_-(v)}{2C_+(v)}|E|,
\eeq
otherwise the set in the left hand side of \eqref{sym_bound} is empty. The function
$$
h(\theta)=v'(\theta)+av'(a\theta+b),
$$
satisfies $|h(\theta_0)|\ge \frac{C_-(v)}{2C_+(v)}|E|$ for some $\theta_0\in [-1/2,1/2]$. Also, $h$ extends to the strip $|\im z|\le 20$ and satisfies $|h(z)|\le h_{\max}(v)$ in that strip for some $h_{\max}(v)>0$. We can pick the bound $h_{\max}(v)$ in such a way that $h_{\max}(v)\ge |h(\theta_0)|$ uniformly in $a,b$, in the range considered. Lemma \ref{harnack} applied to the function
$$
h_1(\theta)=\frac{h(\theta-\theta_0)}{h(\theta_0)}.
$$
with $M=h_{\max}(v)/h(\theta_0)$ and $\lambda=\eta/h(\theta_0)$ yields the following bound
\beq
\label{eq_theta_set_estimate}
|\{\theta\in [-1/2,1/2]\colon |h(\theta)|<\eta\}|\le C \exp\l\{2\frac{\log\eta-\log|h(\theta_0)|}{\log h_{\max}(v)-\log|h(\theta_0)|}\r\}
\eeq
We have
$$
\exp\l\{2\frac{-\log|h(\theta_0)|}{\log h_{\max}(v)-\log|h(\theta_0)|}\r\}=\exp\l\{2-\frac{\log h_{\max}(v)}{\log h_{\max}(v)-\log|h(\theta_0)|}\r\}\le C_1(v),
$$
and
$$
0<\log h_{\max}(v)-\log|h(\theta_0)|\le \max\{\log 2,-\log C_2(v)|E|\}\le -\log C_3(v)|E|,
$$
where $C_3(v)=\min\{C_2(v),1/6\}$ (since we are only considering $|E|\le 3$). Using the last two estimates, \eqref{eq_theta_set_estimate} becomes
$$
|\{\theta\in [-1/2,1/2]\colon |h(\theta)|<\eta\}|\le C_4(v)\eta^{-\frac{2}{\log C_3(v)|E|}},\quad |E|\le 3,
$$
where the implicit dependence on $E$ of the left hand side is in the choice $a,b$ restricted by \eqref{eq_condition_ab}.

To estimate the set in the left hand side of \eqref{sym_bound}, note that the number of intervals of monotonicity of the function $\theta\mapsto v(\theta)+v(a\theta+b)$ is bounded by a constant $M(v)$ that depends only of $v$ (note that this fact is not trivial and is shown in \cite{G}). By considering the sets where $|h(\theta)|<\eta$ and $|h(\theta)|\ge \eta$, we arrive to
$$
|\text{l. h. s. of \eqref{sym_bound}}|\le C_5(v)(\eta^{-1}\delta+\eta^{-\frac{2}{\log C_3(v)|E|}})
$$
(where $M(v)$ is absorbed by $C_5(v)$). Balancing the powers leads to
$$
|\text{l. h. s. of \eqref{sym_bound}}|\le C_6(v)\delta^{\frac{2}{2-\log C_3(v)|E|}},
$$
which implies \eqref{sym_bound}.
\end{proof}
Following \cite{BGS}, define an {\it elementary region} $\Lambda\subset \Z^2$ as a difference of a rectangle and its translation over some non-zero lattice vector. Let $\mathcal{ER}(N)$ be the set of all elementary regions of diameter $N$; the diameter of $\Lambda$ is denoted by $\diam(\Lambda)$. For an elementary region $\Lambda\in \mathcal{ER}(N)$, define the {\it Green's function} of the operator $H(\theta_1,\theta_2)$ restricted to $\Lambda$ in the usual way,
$$
G_{\Lambda}(\theta_1,\theta_2,E)=(H_{\Lambda}(\theta_1,\theta_2)-E)^{-1}=(\l.\one_{\Lambda}(H(\theta_1,\theta_2)-E)\r|_{\ran \one_{\Lambda}})^{-1},
$$
where $\one_{\Lambda}$ is the indicator function of $\Lambda$ (that is, the restriction operator). For $\gamma>0$, define also the set of ``good'' phases,
\beq
\label{greengood0}
\g^{\gamma,b}(\Lambda,E)=\{\theta_1,\theta_2\subset \T^2\colon G_{\Lambda}(\theta_1,\theta_2,E)\text{ satisfies \eqref{greengood1},\eqref{greengood2}}\},\text{ that is,}
\eeq
\beq
\label{greengood1}
\|G_{\Lambda}(\theta_1,\theta_2,E)\|< \lambda^{-1}e^{\sigma(\Lambda)^b}
\eeq
\beq
\label{greengood2}
|G_{\Lambda}(\theta_1,\theta_2,E)(n_1,n_2)|< 
e^{-\gamma|n_1-n_2|}\,\text{ for all }\,n_1,n_2\in \Lambda,\, |n_1-n_2|\ge \frac14 \sigma(\Lambda).
\eeq
Here the norm in \eqref{greengood1} can be chosen to be, for example, Hilbert--Schmidt norm (however, the choice of a particular matrix norm does not matter as one can essentially ignore factors of $N^C$). The complementary set of ``bad'' phases is
\beq
\label{greenbad}
\b^{\gamma,b}(\Lambda,E)=\T^2\setminus \g^{\gamma,b}(\Lambda,E).
\eeq
\begin{theorem}
\label{initial_sym}

Suppose $v$ satisfies the assumptions of Theorem $\ref{initial_ldt}$. Fix $b,\mu\in (0,1)$. Then, for all unit line segments $L\subset \R^2$, we have
$$
|\b^{\gamma,b}(\Lambda,E)\cap L|_1\le \exp\{-\sigma(\Lambda)^{b(1-\mu)}\},
$$
uniformly in $\Lambda\in \ER(N)$ with $\gamma=\frac12 \log\lambda$, assuming 
\begin{multline}
\label{eq_initial_lambdaNE}
\lambda\ge \lambda_0(N,b,v,\mu),\quad  N\ge N_0(b,v,\mu),\\  |E-U_j|\ge \frac{\lambda}{\exp\{\sigma(\Lambda)^{b \mu}\}},\quad j=1,\ldots,N_{\mathrm{int}}.
\end{multline}
\end{theorem}
\begin{proof}
Suppose 
\begin{multline}
\label{v_level}
|v(\theta_1+n_1\omega)+v(\theta_2+n_2\omega)-(E-U_j)/\lambda|>\delta,\\ \text{for all}\,(n_1,n_2)\in \Lambda, \,\,j=1,\ldots,N_{\mathrm{int}}.
\end{multline}
Then, using the resolvent identity (where $\Delta_{\Lambda}$ and $U_{\Lambda}$ denote the restrictions of the corresponding operators)
\begin{multline*}
(H_{\Lambda}(\theta_1,\theta_2)-E)^{-1}=(\lambda V(\theta_1,\theta_2)-E-\Delta_{\Lambda}-U_{\Lambda})^{-1}\\
=\{I-(\lambda V_{\Lambda}(\theta_1,\theta_2)+U_{\Lambda}-E)^{-1}\Delta_{\Lambda}\}^{-1}(\lambda V_{\Lambda}(\theta_1,\theta_2)+U_{\Lambda}-E)^{-1},	
\end{multline*}
we can see (cf. \cite[Lemma 4.1]{BGS}, or by direct expansion of the resolvent)
\beq
\label{resol1}
|(H_{\Lambda}(\theta_1,\theta_2)-E)^{-1}(n_1,n_2)|\le (1-16\lambda^{-1}\delta^{-1})^{-1}(16\lambda^{-1}\delta^{-1})^{|n_1-n_2|+1},
\eeq
and
\beq
\label{resol2}
\|(H_{\Lambda}(\theta_1,\theta_2)-E)^{-1}\|\le 8\delta^{-1}\lambda^{-1},
\eeq
where we used the fact that $\|\Delta_{\Lambda}\|\le \|\Delta\|=4$.
Take $\lambda\ge e^{\sigma(\Lambda)^{b}}$, $\delta=\lambda^{-1/2}$, and suppose that $(\theta_1,\theta_2)$ satisfy \eqref{v_level}. Since $\lambda^{-1}\delta^{-1}= \lambda^{-1/2}$, \eqref{resol1} and \eqref{resol2} imply
$$
|(H_{\Lambda}(\theta_1,\theta_2)-E)^{-1}(n_1,n_2)|\le e^{-\frac12 \log(\lambda)|n_1-n_2|},
$$
$$
\|(H_{\Lambda}(\theta_1,\theta_2)-E)^{-1}\|\le C \lambda^{-1}e^{\frac12 \sigma(\Lambda)^{b}},
$$
which implies \eqref{greengood1},\eqref{greengood2}.
It remains to estimate the measure of the set of $\theta$ for which \eqref{v_level} fails. Let $L \subset \R^2$ be a line segment. Apply Theorem \ref{initial_ldt}:
\begin{multline*}
\l|\l\{(\theta_1,\theta_2)\in L\colon \eqref{v_level} \text{ fails}\r\}\r|_1\le c_1(v)N_{\mathrm{int}}|\Lambda|\exp\l\{-\frac{c_2(v)\log\delta}{\log |c_3(v)E/\lambda|}\r\}\\
\le c_1(v)N_{\mathrm{int}}|\Lambda|\exp\l\{\frac{c_2(v)\log\delta}{\sigma(\Lambda)^{b\mu}}\r\}=c_1(v)N_{\mathrm{int}}|\Lambda|\exp\l\{-\frac12 c_2(v)\sigma(\Lambda)^{b(1-\mu)}\r\}.
\end{multline*}
The constants $c_1(v)$ and $-\frac12 c_2(v)$, as well as the factor $N_{\mathrm{int}}|\Lambda|$, can be absorbed into the exponent, if one chooses a slightly smaller $\mu$ in the beginning.
\end{proof}
\begin{remark}
The parameter $E$ in Theorem \ref{initial_ldt} corresponds to $\lambda^{-1}(E-U_j)$ in Theorem \ref{initial_sym}.
\end{remark}
\section{Some preliminaries}
{\it Convention regarding the constants. }For simplicity of the language, we will often use the following construction: the constant $C$ in $N^C$ will always mean some absolute constant, and the constant does not have to be the same in all claims. For example, a claim ``Suppose $A(N^C)$. Then $B(N^C)$'' will mean the following: for every $C_1$ there exists $C_2$ such that $A(N^{C_1})$ implies $B(N^{C_2})$. Usually, the exact dependence of $C_2$ on $C_1$ will not be important.

We formulate two standard covering lemmas \cite[Lemma 2.2, Lemma 2.4]{BGS} with slight changes in notation that does not affect the proofs.
\begin{prop}
\label{bgs22}
Suppose $\Lambda\subset \Z^2$ is an arbitrary set with the following property: for every $m\in \Z^2$, there is a subset $W(m)\subset \Lambda$ with $m\in W(m)$, $\diam W(m)\le N$, and the Green's function $G_{W(m)}(E)$ satisfying for some $t,N,A>0$
$$
\|G_{W(m)}(E)\|<A,
$$
$$
|G_{W(m)}(E)(m,n)|\le e^{-tN}, \quad \text{for all}\quad n\in \partial_{\ast}W(m),
$$
where $\partial_{\ast}W(m)$ is the internal boundary of $W(m)$ relative to $\Lambda$, that is,
$$
\partial_{\ast}W(m)=\{n\in W(m)\colon \text{ there exists }k\in \Lambda\setminus W(m) \text{ with }|k-n|=1\}.
$$
Then, assuming $4N^2e^{-tN}\le \frac12$, we have
$$
\|G_{\Lambda}(E)\|\le 2N^2 A.
$$
\end{prop}

\begin{prop}
\label{bgs24}
Suppose $M,N$ are positive integers such that for some $0<\tau<1$
$$
N^{\tau}\le M\le 2N^{\tau}.
$$
Let $\Lambda_0\in \ER(N)$ be an elementary region of size $N$ with the property that for all $\Lambda\subset \Lambda_0$, $\Lambda\in \ER(L)$ with $M\le L\le N$, the Green's function $G_{\Lambda}(E)$ satisfies
\beq
\label{bgs28}
\|G_{\Lambda}(E)\|\le e^{L^b}
\eeq
for some fixed $0<b<1$. We say that $\Lambda \in \ER(L)$, $\Lambda\subset\Lambda_0$ is good if, in addition to \eqref{bgs28}, we have the off-diagonal decay, that is,
$$
|G_{\Lambda}(m,n)|\le e^{-\gamma|m-n|}\quad \text{ for all }m,n\in \Lambda,\,\,|m-n|>\frac14 L,
$$
where $\gamma>0$ is fixed. Otherwise, $\Lambda$ is called bad. Assume that any family of pairwise disjoint bad $M'$-regions in $\Lambda_0$ with $M+1\le M'\le 2M+1$, has at most $N^b$ regions in it. Then, under these assumptions, we have
$$
|G_{\Lambda_0}(m,n)|\le e^{-\gamma'|m-n|}\,\quad \text{ for all }m,n\in \Lambda_0,\,\,|m-n|>\frac14 N,
$$
where $\gamma'=\gamma-N^{-\delta}$ and $\delta=\delta(b,\tau)>0$, provided $N\ge N_0(b,\tau,\gamma)$.
\end{prop}
\begin{remark}
While the result of Proposition \ref{bgs24}	is proved for a fixed $\tau$, one can choose $\delta$ and $N_0$ uniformly to serve an interval $\tau\in [\tau_0,\tau_1]\subset(0,1)$ (it is important that the endpoints are separated from $0$ and $1$).
\end{remark}

\subsection{Facts from real algebraic geometry} Similarly to all previously known higher-dimensional localization results, we will actively use real semi-algebraic sets. We assume that the reader is familiar with \cite[Section 7]{BGS} or \cite[Chapter 9]{B1}, which contains the definitions and relevant references. Below, we summarize some facts from real algebraic geometry that will be used during later constructions (starting from the definition). 
\begin{itemize}
	\item[(sa1)] A set $S\subset \R^d$ is called (closed) {\it semi-algebraic} if it is a finite union of sets, each of which is defined by finitely many polynomial inequalities or equalities of the form $Q\ge 0$, $Q\le 0$, or $Q=0$. We say that $\deg S\le sd$, if the set $S$ can be described using $s$ inequalities of the above type, with polynomials $Q$ of degree $\le d$ (and $\deg S$ is defined as the smallest possible value of $sd$ among all such representations). By a slight abuse of notation, bounds of the form $A\le (\deg S)^C$ would actually mean ``there exists an absolute constant $C$ such that the quantity is bounded by $A\le (2+\deg S)^C$'', in order to avoid considering separate case $\deg S=1$ or $\deg S=0$ (we will only be interested in such bounds for $\deg S\gg 1$).
	\item[(sa2)] If $S_1,S_2\subset \R^d$ are semi-algebraic, then $S_1\cup S_2$, $S_1\cap S_2$ are semi-algebraic, each of degree $\le (\deg S_1+1)(\deg S_2+1)$.
	\item[(sa3)] {\it Tarski--Seidenberg principle}: let $S\subset \R^{d+1}$ be semi-algebraic, and $p\colon \R^{d+1}\to \R^d$ be the standard projection. Then $p(S)$ is semi-algebraic, and $\deg p(S)\le (\deg S)^{C(d)}$. In general, $p(S)$ may not be closed, but it will always be closed if $S$ is compact. We will always apply this property to compact semi-algebraic sets.
	\item[(sa4)] Let $S\subset [0,1]^2$ be a semi-algebraic set of zero Lebesgue measure. For any $\varepsilon>0$, the $\varepsilon$-neighborhood of $S$ can be covered by $(\deg S)^{C(d)}\varepsilon^{-1}$ balls of radius $\varepsilon$ with centers on $S$. 
	\item[(sa5)] Let $S\subset [0,1]^d$ be semi-algebraic. Then $\partial S$ is also semi-algebraic, $\deg(\partial S)\le (\deg S)^{C(d)}$. Moreover, $\partial S$ is a union of $\le (\deg S)^{C(d)}$ semi-algebraic sets of dimensions $\le d-1$, whose Hausdorff measures of the corresponding dimensions are bounded by $(\deg S)^{C(d)}$. One can also obtain stronger statements from Proposition \ref{prop_triangulation}.
	\item[(sa6)] A semi-algebraic set $S\subset [0,1]^2$ of Lebesgue measure $\ge \varepsilon$ always contains a ball of radius $(\deg  S)^{-C}\varepsilon$. This property can also be applied on 2D surfaces. The following example will be important. Suppose $\mathcal C\subset [0,1]^2$ is a piecewise algebraic curve of length $L$ with $\le N$ smooth pieces of degree $\le B$ each. Let $I\subset [0,1]$ be an interval of length $L$. Suppose, $\mathcal C\times I\subset[0,1]^3$ is a union of $K$ semi-algebraic subsets of degree $\le B^{C}$. Then there is a curve segment $\mathcal C'\subset\mathcal C$ and a sub-interval $J\subset I$ such that $\mathcal C'\times J$ is contained in one of the $K$ above-mentioned subsets, where $|\mathcal C'|,\,|J|\ge L^{2} B^{-C}K^{-1}$.
	\item[(sa7)] A semi-algebraic set $S\subset \R^d$ has $\le (\deg S)^{C(d)}$ connected components, each of which is a semi-algebraic set of degree $\le (\deg S)^{C(d)}$. If $d=1$, then $S$ is a union of $\le (\deg S)^{C(d)}$ closed line segments.
 	\item[(sa8)] Instead of $\R^d$, one can consider semi-algebraic subsets of $\T^d$ or $\mathbb S^d$, using algebraic local charts (the coordinates induced from $\R^d$ on the torus, or stereographic projection on the sphere), with obvious modifications for previous properties. As an example: let $\mathcal A\subset{\R^{d+1}}$ be semi-algebraic. Fix $\eta>0$, and let $\Xi\subset \mathbb S^d$ be the set of directions in which $\mathcal A$ contains a line segment of length $\ge \eta$. Then $\Xi$ is semi-algebraic, and $\deg \Xi\le (\deg \mathcal A)^{C(d)}$. Similar constructions will be used during the course of the energy elimination part. In particular, an analogue of (sa6) holds for $\mathcal C\subset \mathbb S^2$.
 	\item[(sa9)] Let $S\subset [0,1]^d$ be a (topologically) connected semi-algebraic subset. Then $S$ is algebraically path connected. That is, for $p,q\in S$ there exists a curve $C\subset S$ connecting $p$ and $q$ such that $C$ is a union of $(\deg S)^{C(d)}$ smooth algebraic pieces of degrees $\le (\deg S)^{C(d)}$ and of total length $\le (\deg S)^{C(d)}$ (uniformly in $p,q$).
 	\item[(sa10)] Let $S\subset [0,1]^d$ be a semialgebraic subset, $\diam(S)\le 1$. During some proofs, we will use the following dyadic layer expansion of $S$. Fix $0<\varepsilon<1$. Define
 	$$
 	S(k)=\{x\in S\colon\dist(x,\partial S)=2^{-k}\},\quad k=1,2,3,\ldots,\lceil\log_2(\varepsilon^{-1}) \rceil.
 	$$
Then each $S(k)$ is a semi-algebraic subset (possibly empty),
$$
\deg S(k)\le (\deg S)^C,\quad \dim S_k\le d-1,
$$
and the following relations between neighborhoods are true (here $B_{\varepsilon}$ denotes the open $\varepsilon$-ball about the origin, so that $S+B_{\varepsilon}$ is the $\varepsilon$-neighborhood of $S$):
$$
\bigcup_{k}(S(k)+{B}_{2^{-(k+1)}})\subset S\subset \l(\bigcup_{k}(S(k)+{B}_{2^{-(k+1)}})\r)\cup (\partial S+B_{\varepsilon}).
$$
\end{itemize}
Properties (sa1)--(sa3), (sa4), (sa7) are well known, see references in \cite[Section 7]{BGS}. The remaining properties essentially follow from the quantitative triangulation theorem (Proposition \ref{prop_triangulation}). We include proofs of (sa5), (sa6), (sa8), and (sa9) in the Appendix. Property (sa10) follows from the fact that $\dist(x,\partial S)$ is a semi-algebraic function of $x$, and the level sets of the distance function are always of dimension $\le d-1$.

The following lemma is a modification of the cylindrical decomposition for semialgebraic sets, see, for example, \cite[Section 5.1]{RAG}, with added condition regarding convexity.
\begin{lemma}
\label{convexblocks}
Let $\A\subset [0,1]^2$ be a closed semialgebraic subset of degree $B$. Then $\A$ can be decomposed into a (not necessarily disjoint) union of sets of the following form:
$$
\{(\theta_1,\theta_2)\colon \theta_1\in I,\,\Theta_-(\theta_1)\le \theta_2\le \Theta_+(\theta_1)\}
$$
where $I\subset [0,1]$ is an interval, $\Theta_{\pm}\colon I\to [0,1]$ is a continuous algebraic function on $I$ of degree $\le B^C$, smooth on $I\setminus\partial I$, and either linear, or strictly convex, or strictly concave.
\end{lemma}
\begin{proof}
Using quantitative triangulation theorem (Proposition \ref{prop_triangulation}), one can assume without loss of generality that $\A$ is an algebraic diffeomorphic image of a simplex. Call a point $a\in \partial \A$ singular, if one of the following is true: 
\begin{itemize}
\item $\partial A$ is not $C^{\infty}$-smooth at $a$.
\item The tangent vector to $\partial A$ at $a$ is parallel to one of the coordinate axes, and $a$ is an isolated point of $\partial \A$ with this property.
\item $\partial A$ is $C^{\infty}$-smooth at $a$, but changes convexity at that point.
\end{itemize}
The set of singular points is a finite semi-algebraic set of degree $\le B^C$, and hence contains at most $B^C$ points. Consider the grid formed by horizontal and vertical lines drawn at each singular point. This grid will split $[0,1]^2$ into $\le B^C$ rectangles. Without loss of generality, we can restrict ourselves to one of these rectangles $R$. Since $\A\cap R$ has $\le B^C$ connected components, we can further restrict to a single connected component $\A_1$ of $\A\cap R$. Clearly, the projection of $\A_1$ onto the $\theta_1$ axis is a closed interval $I$. For each $\theta_1\in I$ the set $\{\theta_2\colon (\theta_1,\theta_2)\in \A_1\}$ is also a closed interval, and hence $\A_1$ is the space between two graphs of functions. These two functions must be continuous, but do not have to be smooth due to presence of $\partial R$. However, one can split $I$ into $\le B^C$ intervals of smoothness, in which case they will be of constant convexity on each new interval.
\end{proof}

\section{Arithmetic conditions on frequency and thin semi-algebraic sets}
For a closed set $\A\subset [0,1]^2$, denote by $\eta(\A)$ the length of the longest line segment contained in $\A$ (which exists due to compactness).
\begin{theorem}
\label{arith}
Let $\omega\in [0,1)$ satisfy the following Diophantine condition:
\beq
\label{eq_diophantine}
\|k\omega\|\ge C_{\mathrm{dio}}|k|^{-1-\delta_{\mathrm{dio}}}, \quad 1\le |k|\le N,\quad C_{\mathrm{dio}},\delta_{\mathrm{dio}}>0.
\eeq
Let $\mathcal A\subset[0,1]^2$ be a semi-algebraic set with 
\beq
\label{eq_eta_assumption}
\eta(\A)<\min_{1\le |k|\le 2N}\|k\omega\|.
\eeq
Then
\beq
\label{arith_eq}
\#\{(k_1,k_2)\in \Z^2\colon (\{k_1\omega\},\{k_2\omega\})\in \mathcal A,\,|k_1|,|k_2|\le N\}\le (\deg \A)^C C'_{\omega} N^{3/4+3\delta_{\mathrm{dio}}}.
\eeq
\begin{proof}
We will refer to the set $\{(\{k_1\omega\},\{k_2\omega\}),\,|k_1|,|k_2|\le N\}$ as ``$N$-lattice points'', and the intersection with $\A$ as ``$N$-lattice points on $\A$'', or ``lattice points on $\A$''. Since we can essentially ignore a factor $(\deg \A)^C$, we may, by splitting $\A$ into $\le (\deg \A)^C$ components and possibly switching the variables, make the following reductions:
\begin{enumerate}
\item {\it $\A$ is a single piece described in Lemma $\ref{convexblocks}$. The functions $\Theta_{\pm}$ satisfy $|\Theta_+(\theta_1)-\Theta_-(\theta_1)|\le \eta(\A)$ for all $\theta_1\in I$, and the slope angle of the tangent vector to the graph of each function does not change more than by $\pi/8$ on $I$.} All these assumptions can be obtained from Proposition \ref{convexblocks} by introducing a $(\deg A)^C$ factor, using the assumptions on $\A$.
\item {\it Both functions $\Theta_+$ and $\Theta_-$ are strictly convex on $I$.} Let $R=I\times I'$ be the smallest rectangle that contains the graphs of $\Theta_{\pm}$. If the functions $\Theta_+$ and $\Theta_-$ have opposite convexity or one of them is linear, then one can easily check that $\A\cap R$ will contain a line segment of length at least $\frac13 \diam(R)$, and hence $\diam R\le 3\eta(\A)$. In particular, $R$ contains at most 9 lattice points, and hence one can absorb the contribution from all such pieces or $\A$ into the factor $(\deg \A)^C C_{\mathrm{dio}}'$. The remaining pieces are strictly convex or concave, and can be assumed to be convex without loss of generality.
\item {\it $\mathcal A$ can be covered by $(\deg \A)^C N^{3/4}$ balls of raduis $N^{-3/4}$, with covering multiplicity $\le 100$.} Follows from (sa4) applied to $\partial \A$, as any cover of $\partial \A$ by balls with centers on $\partial \A$ will also cover $\A$. The multiplicity part follows from Vitali covering lemma (one should first cover $\A$ by balls of size $\frac13 N^{-3/4}$).
\item {\it $\A$ cannot contain four distinct $N$-lattice points $p_1,p_2,p_3,p_4$ satisfying} 
$$
p_2-p_1=p_4-p_3.
$$
Indeed, in that case the convex hull of these points must lie above the graph of $\Theta_1$. Moreover, since tangent vector to the graph of $\Theta_2$ cannot change direction by more that $\pi/8$, the graph will not intersect at least one line segment $p_i p_j$ with $i\neq j$, and hence at least one such segment will be completely contained in $\A$, contradicting 
\eqref{eq_eta_assumption}.


\end{enumerate}

We say that $v\in \R^2$ is a {\it short lattice distance vector} if $v=p_1-p_2$ for two distinct $N$-lattice points $p_1$, $p_2$, and $|v|\le 2N^{-3/4}$. Clearly, both components of $v$ cannot exceed $2N^{-3/4}$ by absolute value which, due to Diophantine condition, implies that the number of different short lattice distance vectors is bounded by
\beq
\label{eq_bound_small_distance}
(C_1(\omega)N^{-3/4}N^{1+\delta_{\mathrm{dio}}})^2=C_1(\omega)^2N^{1/2+2\delta_{\mathrm{dio}}}.
\eeq
Let $B$ be one of the covering balls of $\A$ of radius $N^{-3/4}$, obtained in (3), and assume that $B$ contains two $N$-lattice points $p_1,p_2\in \A$. Property (4) implies that $\A$ cannot contain two points with difference vector $p_2-p_1$, unless one of these points coincides with $p_1$ or $p_2$. Using the multiplicity property from (3), we obtain that, for each short lattice difference vector $v$, {\it there are at most $200$ balls that contain some pair of lattice points on $\A$ with difference vector $v$}. Since the number of different short lattice distance vectors is bounded by \eqref{eq_bound_small_distance}, we can conclude that the number of covering balls that contain two or more lattice points on $\A$ is bounded by $200 C_1(\omega)^2 N^{1/2+2\delta_{\mathrm{dio}}}$.

Let us estimate the number of lattice points on $\A$ in each ball. Since $\A$ intersects each vertical line in a segment of length $\le \eta$, there is at most one $N$-lattice point on $\A$ on each vertical line. By counting possible values of the horizontal coordinate, one obtains (similarly to \eqref{eq_bound_small_distance}) the total bound of $C_1(\omega)N^{1/4+\delta_{\mathrm{dio}}}$ lattice points on $\A$ in each $N^{-3/4}$-ball.

It remains to combine the estimates. We split the balls obtained in (3) into two groups. The number of balls that contain at most one lattice point on $\A$ can be simply bounded by the total number of balls, which is $(\deg \A)^C N^{3/4}$.  
For the balls with two or more points, we estimated the number of such balls by $O(N^{1/2+2\delta_{\mathrm{dio}}})$, and the number of lattice points on $\A$ in each ball by $O(N^{1/4+\delta_{\mathrm{dio}}})$. This gives
$$
\#\text{l. h. s. of \eqref{arith_eq}}\le (\deg \A)^C (200 C_1(\omega)^3 N^{1/4+\delta_{\mathrm{dio}}}N^{1/2+2\delta_{\mathrm{dio}}}+N^{3/4}),
$$
which implies the required bound.
\end{proof}
\begin{remark}
\label{remark_large_N}
The constant $C_{\mathrm{dio}}'$ can be bounded by, say, $C (1+C_{\mathrm{dio}}^{-10})$, where $C$ is an absolute constant. One can also replace \eqref{eq_eta_assumption} by $\eta(\A)\le N^{-1-\delta_{\mathrm{dio}}-\varepsilon}$ and the right hand side of \eqref{arith_eq} by $(\deg \A)^{C}N^{3/4+3\delta_{\mathrm{dio}}+\varepsilon}$, by introducing an additional requirement $N\ge N_0(C_{\mathrm{dio}},\varepsilon)$.
\end{remark}
\end{theorem}
\section{Multi-scale induction step and estimate of the Green's function}
Let $H$ be the operator \eqref{h_def} for some fixed background potential $U$, satisfying the complexity bound \eqref{eq_lowcomp_def}. Recall the definitions of good and bad sets for Green's functions,
$$
\g^{\gamma,b}(\Lambda,E)=\{\theta\subset \T^2\colon G_{\Lambda}(\theta_1,\theta_2,E)\text{ satisfies \eqref{greengood1},\eqref{greengood2}}\},\text{ that is,}
$$
$$
\|G_{\Lambda}(\theta_1,\theta_2,E)\|\le \lambda^{-1}e^{\sigma(\Lambda)^b}
$$
$$
|G_{\Lambda}(\theta_1,\theta_2,E)(n_1,n_2)|\le 
e^{-\gamma|n_1-n_2|}\,\text{for all}\,n_1,n_2\in \Lambda,\, |n_1-n_2|\ge \frac14 \sigma(\Lambda).
$$
$$
\b^{\gamma,b}(\Lambda,E)=\T^2\setminus \g^{\gamma,b}(\Lambda,E).
$$
The sets $\g$ and $\b$ depend on $U$. It will be convenient to introduce the smaller set $\g_U^{\gamma,b}(\Lambda,E)$ as follows: we say that $(\theta_1,\theta_2)\in \g_U^{\gamma,b}(\Lambda,E)$ if $G_{\Lambda}(\theta_1,\theta_2,E)$ satisfies \eqref{greengood1}, \eqref{greengood2} {\it for all possible translations of $U$} by vectors from $\Z^2$. For $\Lambda\in \ER(M)$, there are at most $M^{C_{\mathrm{int}}}$ translations that give different $G_{\Lambda}$, which allows to retain the possibility of shifting $U$, while keeping the sets $\b$ exponentially small. Let also
$$
\b_U^{\gamma,b}(\Lambda,E)=\T^2\setminus \g_U^{\gamma,b}(\Lambda,E).
$$
We will often use the following elementary consequence of the resolvent identity.
\begin{prop}
\label{prop_resolvent}
Let $\Lambda \in \ER(N)$, $H_1=\Delta+\lambda V_1$, $H_2=\Delta+\lambda V_2$ are both Schr\"odinger operators on $\ell^2(\Lambda)$, $\|V_j\|_{\infty}\le 2$, $0<b<1$, $\lambda>1$, and
\beq
\label{resolvent_good1}
\|H_1^{-1}\|< \lambda^{-1}e^{\sigma(\Lambda)^b}
\eeq
\beq
\label{resolvent_good2}
|H_1^{-1}(n_1,n_2)|< 
e^{-\gamma|n_1-n_2|}\,\text{ for all }\,n_1,n_2\in \Lambda,\, |n_1-n_2|\ge \frac14 \sigma(\Lambda).
\eeq
Assume also that $\|V_1-V_2\|_{\infty}\le e^{-3\gamma_1 N}$, where $\gamma_1=\max\{\gamma,1\}$, and $N^b\le \frac{1}{10}\gamma_1 N$. Then
$$
\|H_2^{-1}\|< 2\lambda^{-1}e^{\sigma(\Lambda)^b}
$$
$$
|H_2^{-1}(n_1,n_2)|< 
2e^{-\gamma|n_1-n_2|}\,\text{ for all }\,n_1,n_2\in \Lambda,\, |n_1-n_2|\ge \frac14 \sigma(\Lambda).
$$
\end{prop}
\begin{proof}
The resolvent identity implies
$$
H_2^{-1}=H_1^{-1}+\lambda H_1^{-1}(V_1-V_2)H_2^{-1},
$$
or
$$
H_2^{-1}=H_1^{-1}(1-\lambda H_1^{-1}(V_1-V_2))^{-1}.
$$
Using Neumann series, we have
$$
\|H_2^{-1}-H_1^{-1}\|\le \lambda^{-1} e^{N^b}\l(e^{N^b-3\gamma_1N}+e^{2(N^b-3\gamma_1N)}+\ldots\r)\le 2 \lambda^{-1} e^{-2\gamma_1 N},
$$
from which both estimates follow.
\end{proof}
The following is a refined version of the claims of \cite[Remark 3.2, Remark 3.4, Lemma 4.2 and Remark 4.3]{BGS} which state that, essentially, that one can treat sets of parameters $E,\omega,\theta_1,\theta_2$ for which the Green's function $G_{\Lambda}(E,\theta_1,\theta_2)$ is good/bad, as semi-algebraic sets of degree $\le N^C$, where $\Lambda\in \ER(N)$.
\begin{lemma}
\label{semialg_reduction}
Let $\Lambda\in \ER(N)$ be an elementary region. Define
$$
S_{\Lambda}^{\gamma,b}=\{(E,\omega,\theta_1,\theta_2)\colon G_{\Lambda}(E,\theta_1,\theta_2)\text{ does not satisfy \eqref{greengood1} or \eqref{greengood2}}\}.
$$
Then there exists a semialgebraic subset $\A_{\Lambda}^{\gamma,b}\subset \R\times [0,1]^3$, $\deg \A_{\Lambda}^{\gamma,b}\le N^C$, such that
$$
\A_{\Lambda}^{\gamma,b}+B_{e^{-10 \gamma_1 N}}\subset S_{\Lambda}^{\gamma,b},
$$
and for every $(E,\omega,\theta_1,\theta_2)\in \R\times [0,1]\times [0,1]^2\setminus \A_{\Lambda}^{\gamma,b}$ we have
\beq
\label{greengood1_weaker}
\|G_{\Lambda}(\theta_1,\theta_2,E)\|< 8\lambda^{-1}e^{\sigma(\Lambda)^b}
\eeq
\beq
\label{greengood2_weaker}
|G_{\Lambda}(\theta_1,\theta_2,E)(n_1,n_2)|< 
8 e^{-\gamma|n_1-n_2|}\,\text{ for all }\,n_1,n_2\in \Lambda,\, |n_1-n_2|\ge \frac14 \sigma(\Lambda).
\eeq
The same is true if considers the sets in variables $\omega,\theta_1,\theta_2$ with fixed $E$, or with fixed $\omega$, or both.
\end{lemma}
\begin{proof} Without loss of generality, one can assume that $(0,0)\in \Lambda$, since the translation of $\Lambda$ by $(k_1,k_2)$ is equivalent to replacing $(\theta_1,\theta_2)$ by $(\{\theta_1+k_1\omega\},\{\theta_2+k_2\omega\})$. 

We will assume that the potential is of the form 
$$
v_{n_1,n_2}(\omega,\theta_1,\theta_2)=w(\theta_1+n_1\omega,\theta_2+n_2\omega),\quad  (n_1,n_2)\in \Lambda,
$$ 
and $w$ is a real analytic function. Let $w_D(\theta_1,\theta_2)$ be a trigonometric polynomial obtained by truncating the Fourier series of $w$ up to order $D$. Clearly, for $c(v)>0$, we have
$$
\|w_D-w\|_{C^1[0,1]^2}\le e^{-c(w) D}.
$$
Proposition \ref{prop_resolvent} implies that \eqref{greengood1} and \eqref{greengood2} can become worse at most by a factor of $2$, if one replaces $w$ by $w_D$, as long as $D\ge C(w,\gamma)N$. We will now replace each trigonometric term of $w_D$ by a polynominal. It is sufficient to consider terms of the form 
$$
\cos (k\omega+\theta_1)=\cos(k\omega)\cos(\theta_1)-\sin(k\omega)\sin(\theta_1)
$$ 
$|k|\le D$, $|\omega|,|\theta_1|<1$. Let $c_M$ be the $M$-th Taylor polynomial for the cosine. Then
$$
|c_M(k\omega)-\cos(k\omega)|\lesssim \frac{|k|^M}{M!}\le \frac{D^M}{M!},
$$
which can be made small by taking, say, $M\gtrsim D^3$. We thus replace each $v_{n_1,n_2}(\omega,\theta_1,\theta_2)$ by a polynomial $w_{n_1,n_2}(\omega,\theta_1,\theta_2)$ of degree, say, $\le C(w,\gamma) N^{10}$, so that
$$
\|v_{n_1,n_2}-w_{n_1,n_2}\|_{C^1([0,1]^3)}\le e^{-100 \gamma_1 N}.
$$
Now, define $\A_{\Lambda}$ to be the set of parameters for which \eqref{greengood1} or \eqref{greengood2} fail for the operator with potential $w_{n_1,n_2}$, by at least a factor of $4$ in the right hand side. Then, Proposition \eqref{prop_resolvent} and the fact that $v_{n_1,n_2}$ is very close to $w_{n_1,n_2}$, implies the claims of the lemma.
\end{proof}
\begin{remark}
While the constants in the proof depend on $\gamma$, one can easily transfer the dependence on $\gamma$ into the choice of the initial scale; which, in further arguments, will depend on $\gamma$ anyway. One can also replace the constant $8$ in \eqref{greengood1_weaker}, \eqref{greengood2_weaker} by any $C>1$.
\end{remark}

In the following, we will denote by $C'_{\mathrm{int}}$, $C'_{\mathrm{dio}}$ some constants that depend, respectively, only on $C_{\mathrm{int}}$, $C_{\mathrm{dio}}$. We will use this notation in a way similar to $N^C$. Namely, the use of $C'_{\mathrm{int}}$ will mean that the claim is true with some constant that depends only on $C_{\mathrm{int}}$. The actual value of that constant may depend on the context.
\begin{prop}
\label{bgs_44}
Fix $\gamma,C_{\mathrm{dio}},\delta_{\mathrm{dio}},m_{\mathrm{int}},C_{\mathrm{int}}>0$. There exist constants $C'_{\mathrm{int}}$ and $C'_{\mathrm{dio}}$ such that, if $0<b<1$, $0<\rho<1$, $C_1>0$ satisfy
\beq
\label{brho}
b-3/4-3\delta_{\mathrm{dio}}-3\rho>0,\quad C_1>\frac{C'_{\mathrm{int}}+C'_{\mathrm{dio}}+1/\rho}{b-3/4-3\delta_{\mathrm{dio}}-3\rho},
\eeq
then there exists a positive integer $\overline{N}_0=\overline{N}_0(\gamma,b,\rho,C_1,C_{\mathrm{dio}},\delta_{\mathrm{dio}},C_{\mathrm{int}},m_{\mathrm{int}})$ such that the following is true. Suppose $N_0,N_1$ are positive integers,
$$
\overline{N}_0\le 100 N_0\le N_1^{\rho},
$$ 
and for every $N_0\le M\le N_1$ and any elementary region $\Lambda\in \ER(M)$ we have
\beq
\label{stronger_assumption}
\sup\limits_{L\subset \R^2,E\in \R}|\b^{\gamma,b}(\Lambda,E)\cap L|_1\le \exp(-M^{\rho}),
\eeq
where $L$ runs over all unit line segments. Let $N$ satisfy
$$
N_0^{C_1}\le N\le N_1^{\rho C_1},
$$
and assume that $\omega$ satisfies the finite scale Diophantine condition
\beq
\label{eq_diophantine_bgs44}
\|k\omega\|\ge C_{\mathrm{dio}}|k|^{-1-\delta_{\mathrm{dio}}}, 1\le |k|\le 100 N.
\eeq
Then, for all $\Lambda\subset \ER(N)$ and any unit line segment $L\subset \R^2$ we have (for all possible translations of $U$)
\beq
\label{slices}
\sup\limits_{E\in \R}|\{(\theta_1,\theta_2)\in L\colon \|G_{\Lambda}(\theta_1,\theta_2,E)\|>\lambda^{-1}e^{N^b}\}|_1<e^{-N^{3\rho}}.
\eeq
\end{prop}
The line of the argument in the proof is very similar to \cite[Lemma 4.4]{BGS}. We include it in the Appendix for the convenience of the reader. The following is the main result of this section. 
\begin{cor}
\label{green_main}
Fix $C_{\mathrm{dio}},\delta_{\mathrm{dio}},b,\rho,C_1,C_{\mathrm{int}},m_{\mathrm{int}}$ as in Proposition $\ref{bgs_44}$, with extra assumption $C_1>2/\rho$. Fix also $\gamma_0>0$. There exist $\lambda_0=\lambda_0(C_{\mathrm{dio}},\gamma_0,b,\rho,\delta_{\mathrm{dio}},C_1,C_{\mathrm{int}},m_{\mathrm{int}})$ such that, for $\lambda\ge \lambda_0$, one can find $\widetilde{N}_0=\widetilde{N}_0(\lambda,C_{\mathrm{dio}},\gamma_0,b,\rho,\delta_{\mathrm{dio}},C_1,C_{\mathrm{int}},m_{\mathrm{int}})$ so that the following bound \eqref{green_main_eq} is true for all $N\ge \widetilde{N}_0$, all $\omega$ satisfying the finite scale Diophantine condition
\beq
\label{eq_diophantine_green_main}
\|k\omega\|\ge C_{\mathrm{dio}}|k|^{-1-\delta_{\mathrm{dio}}}, 1\le |k|\le 100 N^2,
\eeq
all $\Lambda\in \ER(N)$, and all unit line segments $L$ in $\R^2$:
\beq
\label{green_main_eq}
|\b^{\gamma,b}_U(\Lambda,E)\cap L|_1\le \exp\{-\sigma(\Lambda)^{\rho}\}.
\eeq
For the asymmetric case, one needs in addition to assume
\beq
\label{green_main_eq_energy}
|E-U_j|\ge \frac{\lambda}{\exp\{(\log\lambda)^{1/(\varepsilon C_1)}\}},\quad j=1,\ldots, N_{\mathrm{int}}.
\eeq
where $0<\varepsilon=\varepsilon(b,m_{\mathrm{int}},C_{\mathrm{int}})<1$ is a small constant.
\end{cor}
\begin{proof}
Fix some $\varepsilon=\varepsilon(b,m_{\mathrm{int}},C_{\mathrm{int}})>0$ (the exact choice will be made later in the proof).
There exists a scale $N_0'=N_0'(\gamma_0,b)$ and $\delta=\delta(\gamma_0,b)$, such that Proposition \ref{bgs24} is applicable in the range 
\beq
\label{eq_gamma_tau_choice}
\gamma\in [\gamma_0,2\gamma_0]=[\gamma_0,\gamma_1],\quad \tau\in [1/4,3/4]
\eeq
for all $N\ge N_0'(\gamma_0,b)$ and with $\delta(b,\tau)=\delta$. Choose $\overline{N}_0$ such that Proposition \ref{bgs_44} is applicable with $\gamma=\gamma_0$ and other parameters introduced earlier, and take $\widetilde{N}_0\ge \max\{100 \overline{N}_0,N_0'\}^{1/\varepsilon}$. Suppose also that $\lambda$ and $E$ are chosen to satisfy \eqref{stronger_assumption} with $\gamma=\gamma_1=2\gamma_0$ for all ``initial scales'' $M\in [\widetilde{N}_0^{\varepsilon},\widetilde{N}_0^{C_1}]$. We have yet to show that such choice is possible in line with \eqref{green_main_eq_energy}, which will be discussed in the end of the proof.

Since $C_1> 2/\rho$, Proposition \ref{bgs_44} applied with $N_0=\widetilde{N_0}$, $N_1=[\widetilde{N}_0^{C_1}]$ thus provides a norm bound for Green's function \eqref{slices} for $N\in [\widetilde{N}_0^{C_1},\widetilde{N}_0^{2 C_1}]$. We now need to extend the off-diagonal decay bound \eqref{greengood2} from $[\widetilde{N}_0^{\varepsilon},\widetilde{N}_0^{C_1}]$ (where we have it due to the assumption), into the new interval $[\widetilde{N}_0^{C_1},\widetilde{N}_0^{2 C_1}]$. Let
$$
N\in [\widetilde{N}_0^{C_1},\widetilde{N}_0^{2 C_1}],\quad M_0=[N^{\varepsilon C_1^{-1}}]\in [\widetilde{N}_0^{\varepsilon},\widetilde{N}_0^{2\varepsilon}],\quad M_1=[N^{1/2}].
$$ 
Both scales $M_0$ and $M_1$ satisfy \eqref{stronger_assumption} with $\gamma=\gamma_1$. We follow the argument in \cite[Corollary 4.5]{BGS} and apply Proposition \ref{bgs24} with $M=M_1$ and $N=N$, which agrees with the choice \eqref{eq_gamma_tau_choice}. We will need to verify the main condition in Proposition \ref{bgs24}, that is, that the number of bad boxes of sizes in $[M_1/2,2M_1]$ is bounded by $N^b$. We will do it through an auxiliary scale $M_0$. The following argument is very similar to the proof of Proposition \ref{bgs_44} (see Appendix). Let
\beq
{\mathcal B}=\cup_{M_0+1\le L\le 2M_0+1}\cup_{\Lambda\in \ER(L),\Lambda\subset [-2M_0,2M_0]^2} \b_U^{\gamma_1,b}(\Lambda,E).
\eeq
In view of Lemma \ref{semialg_reduction}, ${\mathcal B}$ can be replaced by a semi-algebraic set $\mathcal A$,  $\deg \mathcal A \le M_0^{C'_{\mathrm{int}}}$, satisfying the following: 
for $(\theta_1,\theta_2)\in [0,1]\setminus\A$, the same conclusion \eqref{greengood1}, \eqref{greengood2} holds as it was for $(\theta_1,\theta_2)\in [0,1]^2\setminus\mathcal B$, with an extra factor of $8$. For any unit line segment $L\subset \R^2$,
$$
|\mathcal A\cap L|_1\le M_0^{C'_{\mathrm{int}}}e^{-M_0^{\rho}}.
$$
Suppose that $\widetilde{N}_0$ is chosen to be large enough (depending on $\varepsilon$), in order to have, for all $M_0$ under consideration,
\beq
M_0^{C'_{\mathrm{int}}}e^{-M_0^{\rho}}\le C_{\mathrm{dio}} N^{-1-\delta_{\mathrm{dio}}}.
\eeq
Then, one can apply Theorem \ref{arith} and obtain
\beq
\label{eq_small_regions}
\#\{n\in [-N,N]^2\colon \theta+n\omega\in \mathcal A\}\le M_0^{C'_{\mathrm{int}}}C_{\mathrm{dio}}' N^{3/4+3\delta_{\mathrm{dio}}}\le C'_{\mathrm{dio}}N^{\varepsilon C_1^{-1}+3/4+3\delta_{\mathrm{dio}}}.
\eeq
where $C'_{\mathrm{int}}$ is a new constant so that the factor $M_0^{C'_{\mathrm{int}}}$ absorbs $(\deg\A)^C$ from Theorem \ref{arith}. The assumption \eqref{brho} provides that one can choose $\varepsilon>0$ in a way that the right hand side of \eqref{eq_small_regions} is dominated by $N^b$; one also needs to make another assumption of largeness of the initial scale, of the same type as in Proposition \ref{bgs24}, which can be taken care of after choosing $\varepsilon$.

The estimate \eqref{eq_small_regions} provides an upper bound on the number of bad $M_0$-regions. Similarly to the derivation \cite[Equation following (4.36)]{BGS}, one can conclude that most of the $M_1$-regions $\Lambda_1\in \ER(M_1)$ satisfy the following:
\beq
\label{eq_m1_regions}
\|G_{\Lambda_1}(\theta_1,\theta_2,E)(m,n)\|\le e^{-\gamma_1|m-n|-C M_0}, \quad \forall m,n\in \Lambda_1,\quad |m-n|>\frac14 M_1.
\eeq
More precisely, an $M_1$ region can violate \eqref{eq_m1_regions} only if there is $M_0$ box with center on that region, violating \eqref{greengood1} or \eqref{greengood2}. Hence, the number of disjoint $M_1$-regions violating \eqref{eq_m1_regions}, is also bounded by $N^b$. This verifies the assumptions of Proposition \ref{bgs24}, which we can now apply and obtain the off-diagonal decay bound \eqref{greengood2} on the scales $[\widetilde{N}_0^{C_1},\widetilde{N}_0^{2 C_1}]$, with 
\beq
\label{eq_new_gamma}
\gamma_2=\gamma_1-M_1^{-\delta}\ge \gamma_1-\widetilde{N}_0^{-C_1\delta/2}.
\eeq
The process can now be repeated with $\widetilde{N}_1=\widetilde{N}_0^2$, and $\widetilde{N}_{j+1}=\widetilde{N}_j^2$, as long as $\gamma_j$ obtained by the analogue of \eqref{eq_new_gamma} stay within the range $[\gamma_0,\gamma_1]$. Due to super-exponential growth of scales $\widetilde{N}_j$, this can be achieved by taking a large initial scale $\widetilde{N}_0$. One can also absorb the extra factor of $8$ from Proposition \ref{semialg_reduction} into a further decrease of $\gamma_j$ similar to \eqref{eq_new_gamma}, which also will not violate $\gamma_j\ge \gamma_0$; note that we only need to do it with respect to \eqref{greengood2}, since \eqref{greengood1} has extra $\lambda^{-1}$ in it, which \eqref{bgs28} does not require.

It remains to explain the choice of the initial scale. Theorem \ref{initial_sym} needs to be applied with the parameter $b'=b(1-\mu)^{-1}$ with some small $\mu$, so that $b'<1$. The condition for the required estimate is 
$$
\lambda\ge \lambda_0(\widetilde{N}_0)\approx e^{\sigma(\Lambda)^{b'}}\approx \exp\{\widetilde{N}_0^{b' C_1}\}.
$$ 
One needs to pick $\widetilde N_0$ in order to satisfy, say, $\log\lambda_0>4\gamma_0$. Then the conclusion of the present Corollary will be true, however, the range of energies would be defined by \eqref{eq_initial_lambdaNE}, whose denominator does not depend on $\lambda$ and therefore is not optimal for large $\lambda$. One can improve it in the following way: if $\lambda$ is larger than $\lambda_0$ that is required for the scheme to work under all previous assumptions on $\widetilde{N}_0$, one can further increase $\widetilde{N}_0$ as much as Theorem \ref{initial_sym} permits for that $\lambda$; one can check that the resulting range of energies will be of the form \eqref{green_main_eq_energy}.
\end{proof}

\section{Line segments in semi-algebraic sets}
In this section, we consider semi-algebraic subsets of $[0,1]^3$ in the variables $(\omega,\theta_1,\theta_2)$. We will establish several estimates on the amount of  long line segments in those sets. Let 
\beq
\label{s_ass1}
S\subset [0,1]^3,\quad \deg S \le B,\quad \dim S\le 2
\eeq
be a two-dimensional semi-algebraic subset. Let also
\beq
\label{s_ass2}
\Seps=S+B_{\varepsilon},\quad 0<\varepsilon<e^{-B^{\rho}},
\eeq
be the open $\varepsilon$-neighborhood of $S$, where $\rho>0$. The coordinates in $[0,1]^3$ will be denoted by $(\omega,\theta_1,\theta_2)$. Suppose that, for any line segment $L\subset \{\omega\}\times [0,1]^2\subset [0,1]^3$, perpendicular to the $\omega$ axis, we have
\beq
\label{s_ass3}
|L\cap S_{2\varepsilon}|_1< e^{-B^{\rho}}
\eeq
(note that the condition involves a larger neighborhood $S_{2\varepsilon}$). Let
$$
e^{(\log B)^{\beta_1}}\le K\le e^{(\log B)^{\beta_2}},\quad 1<\beta_1<\beta_2.
$$ 
Fix some $\nu>0$. For $\xi\in \mathbb S^2$, consider 
$$
A_{\xi}^1=\{(\omega,\theta_1)\colon S_{\varepsilon} \cap ((\omega,\theta_1,0)+[-1,1]\xi\,\text{ contains an interval of size }K^{-\nu}\}\subset [0,1]^2.
$$
$$
A_{\xi}^2=\{(\omega,\theta_2)\colon \Seps \cap ((\omega,0,\theta_2)+[-1,1]\xi\,\text{ contains an interval of size }K^{-\nu}\}\subset [0,1]^2.
$$

Let us call a direction $\xi$ {\it singular}, if the set $\mathrm{Proj}_{\omega}A_{\xi}^i$ contains an interval of size $K^{-1}$ for $i=1$ or $i=2$. For $k=(k_1,k_2)\in [0,K]^2\cap \Z^2$, define 
\beq
\label{xik_def}
\xi_k=\frac{(1,k_1,k_2)}{\sqrt{1+k_1^2+k_2^2}}.
\eeq
\begin{lemma}
\label{elimination_intervals}
Let $\Seps$, $S_{2\varepsilon}$, $K$, $B$, $\nu$, $\rho$ satisfy the above assumptions. Fix $0<c_1<c_2$. There are at most $B^{C(\beta_1,\beta_2,\nu,\rho,c_1,c_2)}$ pairs $k=(k_1,k_2)$ with $c_1 K\le |k|\le c_2 K$ and $\xi_k$ singular.
\end{lemma}
\begin{proof}
Let 
$$
\Xi=\{\xi\colon \mathrm{Proj}_{\omega}A_{\xi}^1\text{ contains an interval of size }1/K\}\subset \mathbb S^2.
$$
Clearly, $\Xi$ is a semi-algebraic set of degree $\le B^C$ (see (sa8)) and $\{\xi_k\colon k\in \mathcal K\}\subset \Xi$. The points $\xi_k$ of the latter set are $K^{-2}$-separated. Since $\Xi$ has at most $B^C$ connected components, we may assume, by taking large enough $C(\beta_1,\beta_2,\nu,\rho,c_1,c_2)$, that there are at least two singular points $\xi_k$ in the same component (which will, in the end of the proof, lead to a contradiction). By (sa9), there is a curve $\mathcal C_0$ in $\Xi$ connecting those two singular points. One can assume that $\mathcal C_0$ consists of $B^C$ smooth algebraic pieces of degree $B^C$, of total diameter $\ge K^{-2}$. Hence, one of the pieces is a smooth algebraic curve $\mathcal C_1$ in $\Xi$ of diameter $\ge B^{-C}K^{-2}$.

Let us make a few reductions with the curve $\mathcal C_1$. By definition, each point or $\mathcal C_1$ defines a direction $\xi$ such that $\mathrm{Proj}_{\omega}A_{\xi}^1$ contains $1/K$-interval. Let us split $[0,1]$ into intervals of size $(2K)^{-1}$. Then, for each $\xi\in \mathcal C_1$, the set $\mathrm{Proj}_{\omega}A_{\xi}^1$ contains at least one interval of the form $[j/(2K),(j+1)/(2K)]$. Let us split $\Xi$ into $2K$ (possibly overlapping) pieces:
$$
\Xi_j=\{\xi\in \Xi\colon \mathrm{Proj}_{\omega}A_{\xi}^1\supset [j/(2K),(j+1)/(2K)]\},\quad j=0,1,\ldots,2K-1.
$$
Clearly, each $\Xi_j$ is also semi-algebraic of degree $B^C$. Hence, there is a curve piece 
$$
\mathcal C_2\subset \mathcal C_1,\quad \diam\mathcal C\ge B^{-C} K^{-3},\quad \mathcal C_2\subset\Xi_j\,\,\text{ for some }\,\,0\le j\le 2K-1.
$$ 
Without loss of generality, one can assume that $\mathcal C_2$ is also smooth and connected (both assumptions lead to a further loss of at most $B^C$). We thus achieved the following: by sacrificing a factor of $B^C K$ in the size of $\mathcal C_1$, the condition ``$\mathrm{Proj}_{\omega}A_{\xi}^1$ contains some interval of size $1/K$ for each $\xi\in \mathcal C_1$'' has been replaced by ``$\mathrm{Proj}_{\omega}A_{\xi}^1$ contains a particular interval of size $1/2K$ for each $\xi\in \mathcal C_2$''. We will call this construction ``freezing'' the interval of $\omega$.

By definition of $A_{\xi}^1$, for each $\omega\in I$ and $\xi\in \mathcal C$, there exists $\theta_1(\xi,\omega)$ such that $(\omega,\theta_1(\xi,\omega))+[0,1]\xi)\cap \E$ contains a line segment 
\beq
\label{segment}
(\omega,\theta_1(\xi,\omega))+I'(\xi,\omega)\xi
\eeq
 of length $K^{-\nu}$. We can repeat the same ``freezing'' procedure with the interval $I'(\xi,\omega)$. That is, split $[0,1]$ into $2K^{\nu}$ segments $J_l$ and 
 define 
 $$
 \mathcal D_l=\{(\xi,\omega)\in \mathcal C_2\times I\colon\, \exists\, \theta_1\colon S_{\varepsilon}\text{ contains \eqref{segment} with }I'(\xi,\omega)\supset J_l \}.
 $$
 Clearly, $\mathcal C_2\times I=\cup_{l} \mathcal D_l$. Note that, while we do not require $\theta_1(\xi,\omega)$ and $I'(\xi,\omega)$ to have any algebraic dependence of $(\xi,\omega)$, the existence of such $\theta_1,I'$ with $I'\supset J_{l}$ is a semi-algebraic condition of degree $\le B^C$ for every $l$, and hence $\mathcal D_l$ are semi-algebraic sets of degree $\le B^C$. Since $\mathcal C_2\times I$ has 2D measure $\ge K^{-4}N^{-C}$, one can apply (sa6) and conclude that there exists a curve segment $\mathcal C_3\subset \mathcal C_2$ and an interval $I_2\subset I$, both of diameter $\ge K^{-4-\nu}B^{-C}$, such that $\mathcal C_3\times I_2$ is completely contained in one of the sets $\mathcal D_l$.

We can summarize the previous paragraph in the following claim. There exists a smooth algebraic curve $\mathcal C_3\subset \Xi$, $\diam \mathcal C_3\ge K^{-5-\nu}$, an interval $I_2\subset [0,1]$, $|I_2|\ge K^{-5-\nu}$, and an interval $J\subset [0,1]$, $|J|\ge \frac12 K^{-
\nu}$ such that, for $(\omega,\xi)\in I_2\times \mathcal C_3$, we have for some $\theta_1\in [0,1]$
\beq
\label{notes43}
(\omega,\theta_1,0)+\xi (J+[0,K^{-7-\nu}])+B_{\varepsilon/10}\subset S_{3\varepsilon/2}.
\eeq
Note that the addition of $[0,K^{-7-\nu}]$ can be achieved by slightly shrinking $J$; we will need it for future convenience. The addition of an $\varepsilon/10$-ball is possible due to replacement of $S_{\varepsilon}$ by $S_{3\varepsilon/2}$. Also recall that, since $K\gg B^C$, any fixed power of $B$ can be absorbed into an extra factor $K^{-1}$.

For a fixed $\xi$, the condition \eqref{notes43} defines a semi-algebraic set $\Omega$ in the $(\omega,\theta_1)$-plane, of degree $\le B^C$, whose projection onto the $\omega$ axis contains $I_2$. Since $\Omega$ has at most $B^C$ connected components, there is a connected component $\Omega_0\subset \Omega$ whose projection onto the $\omega$ axis contains an interval 
$$
I_3\subset I_2,\quad |I_3|\ge B^{-C} |I_2|\ge K^{-6-\nu}.
$$ 
Hence, using (sa9), one can select a (piecewise algebraic of degree $\le B^C$) curve $\mathcal C_{\xi}\subset \Omega_0$ such that 
$$
\mathrm{Proj}_{\omega}\mathcal C_{\xi}=I_3,\quad |I_3|\le \diam \mathcal C_{\xi}\le \mathrm{length}(\mathcal C_{\xi})\le B^C.
$$
Since $\diam \mathcal C_3\ge K^{-5-\nu}$, one can find a large number of separated points on $\mathcal C_3$:
$$
\xi_s\in \mathcal C_3,\quad 0\le s<K^{10+2\nu}; \quad |\xi_s-\xi_{s'}|\ge K^{-20-2\nu},\quad s\neq s'.
$$
For each $\xi_s$, one can construct a curve $\mathcal C_{\xi_s}$ described above. Recall that $|J|\ge \frac12 K^{-\nu}$, $\diam \mathcal C_{\xi_s}\ge K^{-6-\nu}$, which implies
\beq
\label{BN46}
|\mathcal C_{\xi_s}+J \xi_s+B_{\varepsilon/200}|\ge K^{-8-2\nu}\varepsilon.
\eeq
The set in the left hand side is contained in $S_{3\varepsilon/2}$, and $|S_{2\varepsilon}|\le B^C \varepsilon$ (due to (sa4)). Hence, since $K\gg B^C$, 
the sets in the left hand side of \eqref{BN46} must have significant overlap. That is, for some $s\neq s'$, we would have
\beq
\label{eq_lower_measure}
|(\mathcal C_{\xi_s}+J \xi_s+B_{\varepsilon/200})\cap (\mathcal C_{\xi_{s'}}+J \xi_{s'}+B_{\varepsilon/200})|\ge K^{-20-4\nu} \varepsilon.
\eeq
Define the following subset of $\mathcal C_{\xi_s}+J \xi_s$:
$$
W=\{x\in \mathcal C_{\xi_s}+J\xi_s\colon \dist(x,\mathcal C_{\xi_{s'}}+J \xi_{s'})\le \varepsilon/100\}.
$$
Clearly, $W$ is a two-dimensional semi-algebraic set of degree $\le B^C$, and
$$
(\mathcal C_{\xi_s}+J \xi_s+B_{\varepsilon/200})\cap (\mathcal C_{\xi_{s'}}+J \xi_{s'}+B_{\varepsilon/200})\subset W+B_{\varepsilon/100}.
$$
Using (sa4), we have the following bound:
$$
|(\mathcal C_{\xi_s}+J \xi_s+B_{\varepsilon/200})\cap (\mathcal C_{\xi_{s'}}+J \xi_{s'}+B_{\varepsilon/200})|\le B^C\varepsilon |W|_2,
$$
which, combined with \eqref{eq_lower_measure}, implies a lower bound on the 2D measure of $W$ in the right hand side, and a consequent 1D bound (since, by Fubini, $\mathcal C_{\xi}$ cannot be longer than $B^C$):
$$
|W|_2\ge K^{-21-4\nu};\quad |(x+J\xi_s)\cap W|_1\ge K^{-22-4\nu}\,\,\,\text{for some}\,\,x\in \mathcal C_{\xi_s}.
$$
Since both $W$ and $x+J\xi_s$ are semi-algebraic of degree $\le B^C$, the segment $x+J\xi_s$ contains a smaller segment:
$$
x+J_1 \xi_s\subset (x+J\xi_s)\cap W,\quad |J_1|\ge K^{-23-4\nu}.
$$ 
However, each point  $y\in (x+J_1\xi_s)$ is also $\varepsilon/100$-close to $C_{\xi_{x'}}+J \xi_{s'}$. Using \eqref{notes43}, we can conclude that $y$ will not leave $S_{2\varepsilon}$ after a displacement smaller than $K^{-7-\nu}$ in the direction $\xi_{s'}$. Hence, the point $x\in S_{3\varepsilon/2}$ satisfies
\beq
\label{planesegment}
x+J_1 \xi_s+[0,K^{-7-\nu}]\xi_{s'}\subset S_{2\varepsilon},
\eeq
which is a piece of a plane (convex hull of a triangle). The bound $|\xi_s-\xi_{s'}|\ge K^{-20-2\nu}$ implies that the angles of the triangle \eqref{planesegment} cannot be too small, and hence it must contain a horizontal line segment of length, say, $K^{-100-10\nu}\gg e^{-B^{\rho}}$. We have obtained a contradiction with the assumption \eqref{s_ass3}, which completes the proof for the case of $A_{\xi}^1$; the argument for $A_{\xi}^2$ is the same.
\end{proof}
The following lemma establishes a stronger result if one allows for $B^C K$ rather than $B^C$ exceptional directions.
\begin{lemma}
\label{kakeya}
Let $S_{\varepsilon}$ and $S_{2\varepsilon}$ satisfy the assumptions from the beginning of the section. Fix $\nu>0$ and $0<c_1<c_2$. There are at most $B^C K$ pairs $k=(k_1,k_2)$, $c_1 K\le |k|\le c_2 K$, such that $S_{\varepsilon}$ intersects some line segment in the direction $\xi_k$ in a set of 1D measure $\ge K^{-\nu}$.
\end{lemma}
\begin{proof}
For $\xi\in \mathbb S^2$, denote by
$$
M_{\varepsilon}(\xi)=\max_{x\in \R^3}|S_{\varepsilon}\cap (x+[0,1]\xi)|_1
$$
the measure of the largest 1D section of $S_{\varepsilon}$ in the direction $\xi$. We first establish that, for all $\delta>0$,
\beq
\label{kakeya_eq1}
\|M_{\varepsilon}\|_{L^{10/3}(\mathbb S^2)}\le C(\delta) B^C \varepsilon^{1/5+\delta}.
\eeq
Suppose $L\subset \ell\subset S_{\varepsilon}$ is a subset of a line segment $\ell$. Denote by $L_{\varepsilon}$ the round cylinder with axis $L$ and the base being a disc of radius $\varepsilon$:
$$
L_{\varepsilon}=\{x\in \R^3\colon \dist(x,\ell)\le \varepsilon,\,\,\mathrm{Proj}_{\ell}x\in L\}.
$$ 
Clearly, $L_{\varepsilon}\in S_{2\varepsilon}$. Now let $f_{\varepsilon}(\xi)$ be the maximal possible volume of cylinders $L_{\varepsilon}(\xi)$ among all line segments in the direction $\xi$. With this definition, $f_{\varepsilon}(\xi)\le f_{\varepsilon}^{\ast}(\xi)$, where $f_{\varepsilon}^{\ast}(\xi)$ is the Kakeya maximal function of the indicator function of $\one_{S_{2\varepsilon}}$, see \cite{W}. Recall also that (sa4) implies $|S_{2\varepsilon}|\le B^C \varepsilon$. The bound \eqref{kakeya_eq1} now follows from the following bound established in \cite[Theorem 1]{W}:
$$
\|f^{\ast}_{\varepsilon}\|_{L^q(\mathbb S^2)}\le C(\delta)\varepsilon^{1-3/p-\delta}\|\one_{S_{2\varepsilon}}\|_{L^p(\R^3)},\quad 1\le p\le 3, \quad q\le 2p',
$$
with $p=5/2$, $q=10/3$. As a consequence,
\beq
\label{kakeya_eq2}
\int_{\mathbb S^2}M_{\varepsilon}(\xi)^{10/3}\,d\xi \le C(\delta) B^C \varepsilon^{2/3-\delta}.
\eeq
Let $W_0=\{\xi\in \mathbb S^2\colon |M_{\varepsilon}(\xi)|\ge K^{-\nu}\}$.  \eqref{kakeya_eq2} implies the following bound on the surface measure:
\beq
\label{W0_bound}
|W_0|_{2}\le C(\delta) K^{\frac{10}{3}(\nu+1)} \varepsilon^{2/3-\delta},
\eeq
Since $S_{2\varepsilon}$ intersects any line segment in at most $N^C$ intervals, $\xi\in W_0$ implies that $S_{2\varepsilon}$ contains a line segment of length $N^{-C}K^{-\nu}$ in the direction $\xi$. Let $W\supset W_0$ be the set of all directions with the latter property. Clearly, $W$ satisfies \eqref{W0_bound} with extra factor $B^{C}$ in the right hand side; also, $W$ is semi-algebraic of degree $\le B^C$ (see (sa8)). Due to \eqref{W0_bound}, $W$ is contained in $\varepsilon^{1/3-o(1)}$-neighborhood of its boundary, and, due to (sa6) and with a possible loss of another $B^C$ factor, {\it can be replaced by the $\varepsilon^{1/4}$-neighborhood of a smooth algebraic curve $\Gamma$ in $\mathbb S^2$, $\deg\Gamma\le B^C$}. All directions $\xi_k$ that we care about, must be in $W$.

For each $\xi$ in $W$, denote by $\tilde{\xi}$ the point of intersection of the line $\R \xi$ with the plane $\omega=K^{-1}$ in $\R^3$. This map will transform $\Gamma$ and $W$ into $\widetilde{\Gamma}$, $\widetilde{W}$, respectively, and, since $W\subset \Gamma+B_{\varepsilon^{1/4}}$, we have 
$$
\widetilde{W}\subset \widetilde\Gamma+B_{CK\varepsilon^{1/4}}.
$$
Moreover, for each 
$$
\xi_k=\frac{(1,k_1,k_2)}{\sqrt{1+k_1^2+k_2^2}}, \quad c_1 K\le |k|\le c_2 K,
$$
we have 
$$
\dist\l(\l(\frac{k_1}{K},\frac{k_2}{K}\r),\widetilde\Gamma \r)\le CK \varepsilon^{1/4}.
$$
Note also $\diam \widetilde\Gamma\le c_2'$ (where $c_2'$ depends on $c_2$). Using (sa4) and $\varepsilon\ll K^{-4}$, we can cover $\widetilde{W}$ by $B^C K$ (open) squares of size $K^{-1}$. Since each square contains at most one point $(k_1/K,k_2/K)$, this implies the required bound.
\end{proof}

\section{Elimination of double resonances}
In this section, we will assume that $C_{\mathrm{dio}},\delta_{\mathrm{dio}}$ are fixed. Their choice will be explained in the beginning of the next section. Denote the finite-scale set of Diophantine frequencies
$$
\dc(N)=\{\omega\in [0,1]\colon \|k\omega\|\ge C_{\mathrm{dio}}|k|^{-1-\delta_{\mathrm{dio}}},\text{ for }1\le |k|\le N\}\supset \dc_{C_{\mathrm{dio}},\delta_{\mathrm{dio}}}=\bigcap\limits_N \dc(N).
$$
Note that, for finite $N$, we have dropped the dependence on $C_{\mathrm{dio}},\delta_{\mathrm{dio}}$ from the notation. In this section and later, we will assume that $\gamma>10$. The main object of study in this section is the ``bad set'' 
\beq
\label{bad_e}
\mathcal E_N^r=\{(\omega,\theta_1,\theta_2)\colon \exists j\colon (\theta_1,\theta_2)\in \b_U^{\gamma,b}([-N,N]^2,E_j(\omega)),\,\omega\in \dc(N^r)\},
\eeq
where $E_j(\omega)$ runs over all eigenvalues of $H_{[-M,M]^2}(0,0)$ over all possible translations of $U$ and $N\le M\le N^r$. Note that the set of $E_j(\omega)$ depends on $N$, $\omega$, and $r$, and has cardinality $\le N^C$. In this section, the constants may depend on $r$. Here $b,\gamma$ are obtained from Corollary \ref{green_main}, so that, for $N\ge N_0$, $\E_N^r$ satisfies the conclusion of Corollary \ref{green_main}: every section of $\E_N^r$ by a unit line segment perpendicular to $\omega$ axis has measure $\le e^{-N^{\rho}}$ (we can slightly decrease $\rho$ to absorb  the  factor $N^C$ which comes from considering multiple energies $E_j(\omega)$ at the same time).

The set $\E_N^r$ is not semialgebraic. However, as discussed in Lemma \ref{semialg_reduction}, one can replace it by a smaller set $\E_N^{r,\mathrm{alg}}$, such that, for $(\omega,\theta_1,\theta_2)\notin \E_N^{r,\mathrm{alg}}$, the Green's function $G_{[-N,N]^2}(\theta_1,\theta_2,E)$ satisfies \eqref{greengood1}, \eqref{greengood2} with a non-essential loss of factor of $10$. In other words, one treat the points of the larger set $\E_N^r\setminus \E_N^{r,\mathrm{alg}}$ as ``good''. Moreover, Lemma \ref{semialg_reduction} also implies that we can assume
$$
\E_N^{r,\mathrm{alg}}+B_{e^{-10\gamma N}}\subset\E_N^{r}.
$$
In order to meet the assumptions of Lemmas \ref{elimination_intervals} and \ref{kakeya}, we will need some further preparations. Consider the layer expansion of $\E_N^{r,\mathrm{alg}}$ provided by (sa10) with $\varepsilon=e^{-20\gamma N}$. Each dyadic layer $\E_N^{r,\mathrm{alg}}(k)$ (in the notation of (sa10)) is a $2^{-(k+1)}$-neighborhood of some semialgebraic subset, with the property that the neighborhood of double size is still contained in $\E_N^{r,\mathrm{alg}}$. Moreover, these layers cover $\E_N^{r,\mathrm{alg}}$, except maybe for $e^{-20\gamma N}$-neighborhood of the boundary $\partial \E_N^{r,\mathrm{alg}}$. The total number of layers is $N^{C}$, where we are still using the convention that $C$ may depend on $r$. Note that the points in the neighborhood of the boundary will still satisfy \eqref{greengood1} and \eqref{greengood2} with, say, a loss of factor of 20. Hence, we can further decrease $\E_N^{r,\mathrm{alg}}$ and then consider a single dyadic layer $\mathcal L$:
\beq
\label{eq_layer_definition}
\mathcal L=\mathcal S+B_{\delta}, \quad \mathcal L+B_{\delta}=\mathcal S+B_{2\delta}\subset \E_N^{r,\mathrm{alg}},\quad e^{-20 \gamma N}\le \delta\le e^{-N^{\rho}},\quad \dim\mathcal S\le 2.
\eeq
Let us define $\widetilde{\b}_U^{\gamma,b}([-M,M]^2,E)$ to be set of $\theta\in [0,1]^2$ satisfying the same conditions as $\b_U^{\gamma,b}([-M,M]^2,E)$, relaxed by a factor of $100$. That is, for all possible translations of $U$, we assume for $(\omega,\theta_1,\theta_2)\in [0,1]^3\setminus \widetilde{\b}_U^{\gamma,b}([-M,M]^2,E)$:
$$
\|G_{[-M,M]^2}(\theta_1,\theta_2,E)\|\le 100 \lambda^{-1}e^{\sigma(\Lambda)^b}
$$
$$
|G_{[-M,M]^2}(\theta_1,\theta_2,E)(n_1,n_2)|\le 
100 e^{-\gamma|n_1-n_2|}\,\text{for all}\,n_1,n_2\in \Lambda,\, |n_1-n_2|\ge \frac14 \sigma(\Lambda).
$$
The following theorem is the key step in proving localization. It shows that removal of a small set of frequencies will exclude double resonances.
\begin{theorem}
\label{elimination_main}
Fix $0<c_1<c_2$. There exists $N_0=N_0(r,\lambda,v,\gamma,b,\rho,c_1,c_2,C_{\mathrm{int}},m_{\mathrm{int}})$ such that, for 
\beq
\label{eq_klm}
N\ge N_0,\quad K=e^{(\log N)^{2/\rho}},\quad M=[(\log N)^{3/2\rho}], 
\eeq
there exists 
\beq
\label{dc_measure}
\Omega^{\mathrm{bad}}_N\subset \dc(N^r),\quad |\Omega_N|\le e^{-\frac12 (\log N)^{3/2}},
\eeq
such that, for every $\omega\in \dc(N^r)\setminus \Omega^{\mathrm{bad}}_N$ and $c_1 K\le |k|\le c_2 K$, the following is true. Suppose $E=E_j(\omega)$ in the notation of \eqref{bad_e}. Then 
$$
(\omega,k_1\omega,k_2\omega)\notin \widetilde{\b}_U^{\gamma,b}([-M,M]^2,E)\cap \widetilde{\b}_U^{\gamma,b}([-N,N]^2,E).
$$ 
In other words, either $M$-box or $N$-box around the point $(k_1,k_2)$ is good at the energy $E$.
\end{theorem}
Theorem \ref{elimination_main} establishes absence of double resonances. If the energy is close to an eigenvalue of $N$-box centered at the origin, then, at the distance $\approx K$ from the origin, any point is a center of a good box with respect to that energy. The technical difference is that the size of the good box is variable, although it does not affect the proof too much.
\begin{proof} Ideally, we would like to show that one can remove a small set of frequencies $\omega$ such that, for all $k_1,k_2$ in the considered range, the vector $(\omega,\{k_1\omega\},\{k_2\omega\})$ avoids the set $\E_N^{r,\mathrm{alg}}$ defined earlier in this section. Using the techniques from Section 7, we will be able to establish this for most pairs $(k_1,k_2)$. To deal with the remaining pairs, we would have to consider smaller boxes of size $M$.

Let $L=\lfloor 3 c_2 K +1\rfloor$. Let us split $[0,1]$ into $L$ intervals of size $L^{-1}$, and take $\omega=\omega_0+\Delta\omega$, $|\Delta\omega|<L^{-1}$, $\omega_0=l/L$. Each interval contains at most $3$ discontinuity points of functions $\{k_1\omega\}$, $\{k_2\omega\}$. Hence, it defines at most three line segments in $[0,1]^3$ in the direction $\xi_k$, each of the form 
\beq
\label{w0_segment}
(\omega_0+\Delta\omega,k_1\Delta\omega+\theta_1,k_2\Delta\omega+\theta_2),
\eeq
parametrized by $\Delta\omega$. 

Let $\mathcal L$ be defined as in \eqref{eq_layer_definition}. Suppose one of the segments \eqref{w0_segment} has large intersection with $\mathcal L$ (that is, the intersection contains a segment of size $K^{-\gamma}$). Let us continue the line containing this segment.
Since either $|k_1|\ge \frac12 c_1 K$ or $|k_2|\ge \frac12 c_1 K$, the intersection of that line with one of the planes $\theta_1=0$ or $\theta_2=0$ has coordinates
$$
(\omega_0+O(1)K^{-1},0,\theta_2')\quad \text{or}\quad (\omega_0+O(1)K^{-1},\theta_1',0),
$$
for some $\theta_1'$, $\theta_2'$. Hence, for each $\omega_0$ such that the segment \eqref{w0_segment} has large intersection with $\mathcal L$, the semi-algebraic set 
\beq
\label{proj_union}
\mathrm{Proj}_{\omega}A_{\xi_k}^1 \cup \mathrm{Proj}_{\omega}A_{\xi_k}^2\subset [0,1]
\eeq
defined in Lemma \ref{elimination_intervals}, contains a point $\omega_0+O(1)K^{-1}$. Since the set \eqref{proj_union} has degree $\le N^C$, there are two possibilities for each pair $k=(k_1,k_2)$:
\begin{itemize}
	\item There are $\le N^C$ possible values of $l$ with $\omega_0=l/K$ such that the set \eqref{proj_union} contains a point $\omega_0+O(1)K^{-1}$.
	\item If there are too many possible values of $l$, then a connected component of one of the sets \eqref{proj_union} contains two points $l_1/K$ and $l_2/K$ with $l_1\neq l_2$, and therefore contains an interval of length $\ge K^{-1}$. In this case, the pair $(k_1,k_2)$ is singular in the sense of Lemma \ref{elimination_intervals} with $S_{\varepsilon}=\mathcal L$, $\varepsilon=\delta$ defined in \eqref{eq_layer_definition}.
\end{itemize}
Using Lemmas \ref{elimination_intervals} and \ref{kakeya}, one can split the pairs $(k_1,k_2)$ and intervals \eqref{w0_segment} into the following groups.
\begin{enumerate}
	\item The pairs $(k_1,k_2)$ such that none of the layers $\mathcal L$ contain intervals of length $K^{-\nu}$ in the direction $\xi_k$ (that is, there is no layer for which the condition in Lemma \ref{kakeya} holds). Each of those pairs and each of the corresponding $3L$ line segments \eqref{w0_segment} can be avoided by removing a set of frequencies of measure $\le N^C K^2 L K^{-\nu-1}$, which can be made small by choosing a large $\nu$. This settles the claim for the majority of pairs $(k_1,k_2)$, except for at most $N^C K$ pairs.
	\item At most $N^C K$ pairs that satisfy the condition in Lemma \ref{kakeya} for some layer $\mathcal L$ and are non-singular with respect to the same layer (in the sense of Lemma \ref{elimination_intervals}). The latter implies that each pair gives only $\le N^C$ values of $\omega_0$, and hence $\le N^C$ problematic intervals \eqref{w0_segment}.
	\item At most $N^C$ pairs $(k_1,k_2)$ which are singular with respect to some layer.
\end{enumerate}
To deal with (2) and (3), we will use $M$-boxes. By perturbation of the diagonal entries,
$$
|E_j(\omega)-E_j(\omega_0)|\le 4\lambda \frac{N^r}{L}.
$$
(assuming we are taking the eigenvalues in the same order and are considering the same translations of $U$). 
If a box of size $M$ satisfies \eqref{greengood1}, \eqref{greengood2}, then both of these properties are preserved (up to a factor of 10) under a perturbation of energy of size $\le e^{-100 \gamma M}$, see Proposition \ref{prop_resolvent}. Hence, in order to ``freeze'' the dependence on $\omega$ and replace $E_j(\omega)$ by $E_j(\omega_0)$, it would be sufficient to impose
\beq
\label{freeze1}
4\lambda \frac{N^r}{L}\le e^{-100\gamma M},
\eeq
which is satisfied with our choice \eqref{eq_klm} of $K,L,M$. Suppose now that $\omega=\omega_0+\Delta\omega$. The condition \eqref{freeze1} implies that, instead of considering $G_{[-M,M]^2}(\omega,k_1\omega,k_2\omega,E_j(\omega))$, one can consider $G_{[-M,M]^2}(\omega_0,k_1\omega,k_2\omega,E_j(\omega_0))$. This Green's function satisfies the statement of Corollary \ref{green_main}, and hence, the bad set for this function can be avoided by removing $C K^{-1} M^C e^{-M^{\rho}}$ set of $\omega$ from each interval $[\omega_0,\omega_0+1/L)$. For the singular pairs (3), this gives a total of $C N^C M^C e^{-M^{\rho}}$ removed frequencies. For non-singular pairs (2), we have $K N^C M^C K^{-1} e^{-M^{\rho}}$ (there are $N^C K$ pairs remaining, and $N^C$ intervals for each pair). In both cases, the removed set of frequencies has measure bounded by, say, $e^{-\frac23 M^{\rho}}$, which is in line with \eqref{dc_measure}.
\end{proof}
\begin{remark}
\label{fullmeasure_rem1}
Theorem \ref{elimination_main} is used in the next section to establish localization for zero phase. One can formulate an analogue for arbitrary phase, with the same proof. Fix $\theta_1',\theta_2'$ and, similarly to \eqref{bad_e}, define
$$
\mathcal E_N^r(\theta_1',\theta_2')=\{(\omega,\theta_1,\theta_2)\colon \exists j\colon (\theta_1,\theta_2)\in \b_U^{\gamma,b}([-N,N]^2,E_j(\omega,\theta_1',\theta_2')),\,\omega\in \dc(N^r)\},
$$
where $E_j(\omega,\theta_1',\theta_2')$ are eigenvalues of $H_{[-M,M]^2}(\theta_1',\theta_2')$ instead of $H_{[-M,M]^2}(0,0)$. Then the conclusion of Theorem \ref{elimination_main} holds for $\mathcal E_N^r(\theta_1',\theta_2')$, with the same bounds, but $\Omega^{\mathrm{bad}}_N$ will now depend on $(\theta_1',\theta_2')$. For the case of general phases, we will use the notation $\Omega^{\mathrm{bad}}_N(\theta_1',\theta_2')$.
\end{remark}
\begin{cor}
\label{fullmeasure_cor1}
Under the assumptions of Theorem $\ref{elimination_main}$, the set of $(\theta_1',\theta_2',\omega)$ satisfying $\omega\in \Omega^{\mathrm{bad}}_N(\theta_1',\theta_2')$ has measure $\le e^{-\frac14 (\log N)^{3/2}}$. Moreover, if $L\subset[0,1]^2$ is a unit line segment, then the set of $(\omega,\theta_1',\theta_2')$ with $(\theta_1',\theta_2')\in L$ and $\omega\in \Omega^{\mathrm{bad}}_N(\theta_1',\theta_2')$, has 2D measure $\le e^{-\frac14 (\log N)^{3/2}}$ in $[0,1]\times L$.
\end{cor}
\begin{proof}
Both follow from the Fubini theorem and Markov's inequality, as the sets under consideration have small sections by any line $(\theta_1',\theta_2')=\const$. One can check that measurability follows from the construction in Theorem \ref{elimination_intervals}.
\end{proof}
\begin{remark}
\label{fullmeasure_rem2}
A short range interaction $U(n_1-n_2)=f(n_1-n_2)$ has translation invariance in the diagonal direction. That is, it is natural to consider $\theta=\frac12(\theta_1+\theta_2)$, $\eta=\frac12 (\theta_1-\theta_2)$, make $\theta$ the ergodic parameter, and $\eta$ an external parameter. In this case, the conclusion of Corollary \ref{fullmeasure_cor1} holds for any fixed $\eta$ and shows that the set of ``bad'' pairs $(\omega,\theta)$ has small measure in $[0,1]^2$.
\end{remark}
\section{Proof of localization}
Once all ingredients are in place, the proof of localization is fairly standard. Fix some $0<\delta_{\mathrm{dio}}<1/24$ and $\gamma>10$. Let $\varepsilon_{\mathrm{freq}}>0$. Pick $C_{\mathrm{dio}}=C_{\mathrm{dio}}(\varepsilon_{\mathrm{freq}})>0$ such that
$$
|\dc_{C_{\mathrm{dio}},\delta_{\mathrm{dio}}}\cap[0,1]|\ge 1-\varepsilon_{\mathrm{freq}}.
$$
Choose other parameters in a way that \eqref{brho} is satisfied and, based on that choice, choose a large $\lambda_0$ so that the conclusion of Corollary \ref{green_main} is satisfied. Fix some initial phase $(\theta_1,\theta_2)$. Consider the sets $\Omega^{\mathrm{bad}}_N(\theta_1,\theta_2)$ defined in Theorem \ref{elimination_main} and Remark \ref{fullmeasure_rem1}. Assume that
\beq
\label{fullmeasure}
\omega\in \Omega(U,\lambda,\theta_1,\theta_2):=\dc_{C_{\mathrm{dio}},\delta_{\mathrm{dio}}}\setminus \limsup_{N\to \infty}\Omega^{\mathrm{bad}}_N(\theta_1,\theta_2)=\dc_{C_{\mathrm{dio}},\delta_{\mathrm{dio}}}\setminus \bigcap\limits_{N_0\ge 1}\bigcup\limits_{N\ge N_0}\Omega^{\mathrm{bad}}_N (\theta_1,\theta_2).
\eeq
Choose $N$ sufficiently large enough in order to apply Corollary \ref{green_main}, both for the scales $M$ and $N$ from Theorem \ref{elimination_main}.

We will prove Theorems \ref{main_asym} and \ref{main_sym} at the same time, by showing that the operator $H(\theta_1,\theta_2)$ has Anderson localization, as long as Corollary \ref{green_main} is applicable. In the case of Theorem \ref{main_sym}, this would mean extra assumption \eqref{energyrange}.

Suppose that $\psi$ is a generalized eigenvector of $H(\theta_1,\theta_2)$, that is, a formal solution of the eigenvalue equation
\beq
\label{generalized_eif}
H(\theta_1,\theta_2)\psi=E\psi,\quad \psi(0)=1,\quad |\psi(n)|\le C_{\psi}(1+|n|). 
\eeq
Our goal is to show that $\psi\in \ell^2(\Z^2)$. At this point, we are allowed to pick large $N$ depending, in particular, on $\psi$ and $\omega$. We will always assume that $N$ and $M$ are large enough, so that $\Omega\notin \Omega_l^{\mathrm{bad}}(\theta_1,\theta_2)$ for all $l\ge \min(M,N)$.

From Proposition \ref{arith}, there exists a large constant $r_0$ (independent of $N$ for $N$ large) such that, among $n\in[-N^{r_0},N^{r_0}]^2$, there are at most $N^{7r_0/8}$ values of $n$ such that the box $[-N,N]+n$ fails \eqref{greengood1} or \eqref{greengood2} with a factor of 100. Hence, there exists $R\in [N^{r_0/2},N^{r_0}]$ such that the annulus 
\beq
\label{eq_annulus}
A=[-R-N^{r_0/4},R+N^{r_0/4}]^2\setminus [-R+N^{r_0/4},R-N^{r_0/4}]^2
\eeq
consists of good points $(n_1,n_2)$, that is, the Green's function of any box of size $N$ centered in $A$ satisfies \eqref{greengood1}, \eqref{greengood2} with a factor of 100, for all possible translations of $U$.

We will use Poisson's formula. Suppose, $\psi$ satisfies \eqref{generalized_eif}. If $\Lambda$ is an elementary region and $m\in \Lambda\setminus\partial \Lambda$, then
\beq
\label{poisson}
\psi(m)=\sum\limits_{n\in \Lambda,\,n'\in \Z^2\setminus\Lambda,\,|n-n'|=1}G_{\Lambda}(\theta_1,\theta_2,E)(m,n)\psi(n').
\eeq
Suppose $m\in \partial[-R,R]^2$. Apply \eqref{poisson} with $\Lambda=m+[-N,N]$:
\beq
\label{poisson_result1}
|\psi(m)|\le 100 e^{-\gamma N}\sum\limits_{n'\in \partial (\Z^2\setminus\Lambda)}|\psi(n')|,
\eeq
and keep repeating it for each $\psi(n')$. The size of $A$ guarantees that one can apply $\eqref{poisson_result1}$ at least $N^{r_0/5}$ times without leaving the annulus $A$, which brings an estimate
\beq
\label{poisson_result2}
|\psi(m)|\le C_{\psi}(1000 N)^{N^{r_0/5}} e^{-\gamma N \cdot N^{r_0/5}}(1+N^{r_0}).
\eeq
Clearly, the exponential factor dominates all other factors for large $N$ (depending on $\psi$). This implies
$$
|\psi(m)|\le e^{-\frac12 N\cdot N^{r_0/5}},\quad m\in\partial[-R,R]^2,\quad N^{r_0/2}\le R\le N^{r_0},\quad N\ge N_0(C_{\psi}),
$$
and therefore
\beq
\label{eq_dist_spec}
\dist(E,\sigma(H_{[-R,R]^2}(\theta_1,\theta_2)))\le e^{-N^2}.
\eeq
Let $E_0$ be the closest to $E$ eigenvalue of $H_{[-R,R]^2}(\theta_1,\theta_2)$, and suppose $\omega\in\dc(N^2)\setminus\Omega_N(\theta_1,\theta_2)$ (which is valid due to \eqref{fullmeasure}, maybe after increasing $N$). Theorem \ref{elimination_main} implies that, for $c_1 K \le |k|\le c_2 K$, the Green's function $G_{\Lambda}(\theta_1,\theta_2,E_0)$ satisfies \eqref{greengood1} and \eqref{greengood2} with a factor of $100$, for $\Lambda=k+[-N,N]^2$ or $\Lambda=k+[-M,M]^2$. Resolvent identity from Proposition \ref{prop_resolvent} together with \eqref{eq_dist_spec} implies that \eqref{greengood2} will also hold at the energy $E$ with a factor of 200, which can be summarized as follows (at this point, we only care about the off-diagonal decay):
$$
|G_{\Lambda}(\theta_1',\theta_2',E)(n_1,n_2)|< 
200 e^{-\gamma|n_1-n_2|}\,\text{ for all }\,n_1,n_2\in \Lambda,\, |n_1-n_2|\ge \frac14 \sigma(\Lambda),
$$
$$
\Lambda=k+[-N,N]^2\text{ or }\Lambda=k+[-M,M]^2,
$$
\beq
\label{k_range}
c_1 K\le |k|\le c_2 K,\quad K=e^{(\log N)^{2/\rho}},\quad M=[(\log N)^{3/2\rho}],\quad 0<\rho<1.
\eeq
Clearly, the range of $k$ specified in \eqref{k_range} covers the whole $\Z^d$ except for finitely many points, as long as $c_2>c_1>0$. We will choose a smaller annulus by taking 
$$
c_1< c_1+\beta<c_1'<c_2'<c_2-\beta<c_2
$$ and will estimate $\psi(k)$, $c_1'K\le |k|\le c_2'K$, in the same way as we obtained \eqref{poisson_result2}; however, one needs to be more careful about where to stop the expansion, as the box size is now variable. We will repeat applying \eqref{poisson} until the total exponential factor gained from iterations of \eqref{poisson_result1} exceeds $e^{-\gamma \beta K}$ (which guarantees that we are staying within the annulus). Thus, at the end of the expansion, the ``depth'' of each term is determined by the number of $N$-boxes along the history of applying \eqref{poisson} in order to obtain that term. We will provide a coarse bound: for each $0\le l\le \beta K/N$, we will assume that each term had exactly $l$ $N$-boxes along the way, and then take the sum over $l$.

Suppose that $l=0$, so that each time we have to use an $M$-box. Then, the total number of steps is $[\beta K/M]$ (rounding errors can be easily absorbed into the final estimate), which leads to the contribution of these terms into $\psi(k)$ estimated by
\beq
\label{eq_all_mboxes}
(1000M)^{\beta K/M}\exp\{-\gamma \beta K\}C_{\psi}(1+|K|)\le e^{-\frac12 \gamma \beta K},\,\, \text{ assuming }\,\,\frac{\log M}{M}\ll \gamma,
\eeq
which implies, under the assumptions that all boxes are $M$-boxes, that
\beq
\label{psi_exp}
|\psi(k)|\le C_{\psi}\exp\{-\frac12 \gamma\beta c_1'c_1^{-1}|k|\}.
\eeq
Now, suppose there is only one $N$-box along the course of applying \eqref{poisson}, and all other boxes are $M$-boxes. The total contribution from these configurations has the same exponential factor (by construction), has the factor $(1000M)^{N/M}$ replaced by a factor of $1000N$, and adds an extra combinatorial factor $K/M$. Thus, \eqref{eq_all_mboxes} at $l=1$ is bounded by \eqref{eq_all_mboxes} at $l=0$, multiplied by the factor
$$
1000 M^{-1} KN (1000M)^{-N/M}\ll 1.
$$
Each additional $N$-box generates an extra similar (in fact, smaller) factor, which implies that the contribution from the terms with $l>0$ is dominated by the term with $l=0$, and \eqref{psi_exp} is true in all cases.
\begin{remark}
\label{full_measure}
Using Corollary \ref{fullmeasure_cor1}, one can establish Anderson localization for a full measure set of $(\theta_1,\theta_2)$, or for a subset of a line segment $L$ of full 1D measure, by considering the sets of $(\theta_1,\theta_2,\omega)$ constructed in Corollary \ref{fullmeasure_cor1} instead of sets of $\omega$ at fixed $(\theta_1,\theta_2)$. None of the three cases (fixed phase, full measure subset of $[0,1]^2$, full measure subset of $L$) seem to directly imply one another. 

One can also approach this argument from the measure-theoretic point of view. We have essentially shown that, for each $(\theta_1,\theta_2)\in [0,1]^2$, the set of frequencies $\omega$ such that the operator $H(\theta_1,\theta_2)$ has Anderson localization for all possible translations of $U$, has full measure in $\dc_{C_{\mathrm{dio}},\delta_{\mathrm{dio}}}$, and hence the set of $(\omega,\theta_1,\theta_2)$ with this property has full measure in $\dc_{C_{\mathrm{dio}},\delta_{\mathrm{dio}}}\times [0,1]^2$, which allows to apply Fubini's theorem. However, the argument in the previous paragraph is useful in case we do not have an explicit arithmetic condition modulo a zero measure set of frequencies.
\end{remark}
\section{Appendix}

For the convenience of the reader, we include the proof of Proposition \ref{bgs_44}, mostly following \cite{BGS} with appropriate modifications. As stated in \cite{BGS}, the last steps of the proof can be replaced by an application of the following version of Cartan's lemma, which became standard in more recent papers. For the proof, see \cite[Proposition 14.1]{B1}, or \cite[Lemma A.2]{B4}:
\begin{prop}
\label{prop_cartan}
Let $A(\sigma)$ be a self-adjoint $D\times D$ matrix function of a real parameter $\sigma\in [-\delta,\delta]$, satisfying the following conditions:
\begin{enumerate}
	\item $A(\sigma)$ is real analytic in $\sigma$, and admits a holomorphic extension to the strip $(-\delta_1,\delta_1)+i(-\delta_2,\delta_2)$, satisfying in that strip
$$
\|A(z)\|\le B_1.
$$
\item For each $\sigma \in [-\delta_1,\delta_1]$, there is a subset $\Lambda\subset[1,D]$, such that
$$
|\Lambda|<D_0,\quad \|\l(R_{[1,N]\setminus\Lambda}A(\sigma)R_{[1,N]\setminus\Lambda}\r)^{-1}\|<B_2.
$$
\item $|\{\sigma\in [-\delta_1,\delta_1]\colon \|A(\sigma)^{-1}\|>B_3\}<10^{-3}\delta_2 (1+B_1)^{-1}(1+B_2)^{-1}|$.
\end{enumerate}
Then, for any $\varkappa<(1+B_1+B_2)^{-10 D_0}$, we have
\beq
\label{eq_cartan}
|\{\sigma\in (-\delta_1/2,\delta_1/2)\colon \|A(\sigma)^{-1}\|>\varkappa^{-1}\}|<\exp\frac{c\log \varkappa}{D_0\log (D_0+B_1+B_2+B_3)}.
\eeq
\end{prop}

\subsection{Proof of Proposition \ref{bgs_44}} Choose some $M\in [N_0,N_1]$ and define $N=[M_0^{C_1}]$, where $C_1$ will be a large constant. Let $\Lambda_0\in \ER(N)$. Consider the partition
\beq
\label{eq_partition}
\Lambda_0=\cup_{\alpha}\Lambda_{\alpha},\quad \Lambda_{\alpha}=Q_{\alpha}\cap \Lambda_0,
\eeq
where
$$
Q_{\alpha}\in [-M_0,M_0]^2+2M_0\Z^2
$$
is a translation of an $M_0$-cube by an integer vector multiple of $2M_0$, and the union in \eqref{eq_partition} runs over non-empty $\Lambda_{\alpha}$. It is explained in \cite{BGS} that $\Lambda_{\alpha}\in \ER(M')$, for some $M_0\le M'\le 2M_0$, except maybe for at most five values of $\alpha$.

Let
\beq
\mathcal B=\cup_{M_0\le M\le 2M_0}\cup_{\Lambda\in \ER(M),\Lambda\subset [-M,M]^2} \b_U^{\gamma,b}(\Lambda,E).
\eeq
In view of Lemma \ref{semialg_reduction}, one can replace ${\mathcal B}$ by a semi-algebraic set $\mathcal A$,  $\deg \mathcal A \le M_0^{C'_{\mathrm{int}}}$, satisfying the following: for any unit line segment $L\subset \R^2$,
$$
|\mathcal A\cap L|_1\le M_0^{C'_{\mathrm{int}}}e^{-M_0^{\rho}}.
$$
The replacement of $\mathcal B$ by $\mathcal A$ will add a (non-essential for the argument) factor 16 to \eqref{eq_green_cartan1}. Suppose that $\overline {N}_0$ is chosen to be large enough, in order to have, for all $N$ under consideration,
\beq
\label{eq_n_req1}
M_0^{C'_{\mathrm{int}}}e^{-M_0^{\rho}}\le C_{\mathrm{dio}} N^{-1-\delta_{\mathrm{dio}}}.
\eeq
Then, one can apply Theorem \ref{arith} and obtain
$$
\#\{n\in [-N,N]^2\colon \theta+n\omega\in \mathcal A\}\le M_0^{C'_{\mathrm{int}}}C_{\mathrm{dio}}' N^{3/4+3\delta_{\mathrm{dio}}},
$$
where $C'_{\mathrm{int}}$ is a new constant so that the factor $M_0^{C'_{\mathrm{int}}}$ absorbs $(\deg\A)^C$ from Theorem \ref{arith}. Let us call $\Lambda_{\alpha}$ {\it good} if $\theta+n\omega\notin \A$ for all $n\in \Lambda_{\alpha}$, and otherwise call $\Lambda_{\alpha}$ bad. Let
$$
\Lambda_{\ast}=\cup_{\Lambda_{\alpha}\text{ is bad}}\Lambda_{\alpha}.
$$
We have
$$
\#\Lambda_{\ast}\le M_0^{C'_{\mathrm{int}}} C_{\mathrm{dio}}' N^{3/4+3\delta_{\mathrm{dio}}}.
$$
(here $C'_{\mathrm{int}}$ was increased by 2). Now, one can apply the same derivation as in \cite[equation (4.17)]{BGS} and obtain
\beq
\label{eq_green_cartan1}
\|G_{\Lambda_0\setminus\Lambda_{\ast}}(\theta_1,\theta_2,E)\|\le 16\lambda^{-1}M_0^2 e^{(2 M_0)^b},
\eeq
assuming $M_0^2 e^{-\gamma M_0/2}\le 1/2$ (which can also be achieved by fixing a large $\overline{N}_0$).
The bound \eqref{eq_green_cartan1} holds for all $\theta\in [0,1]^2$, with $\Lambda_{\ast}$ depending on $\theta$. The factor $16$ comes from the factor $8$ from Proposition \ref{semialg_reduction}.


In order to satisfy the second assumption of Cartan's lemma, we will need to introduce another scale: $M_1=[(10 M_0)^{1/\rho}]$. Suppose, all boxes $[-M_1,M_1]^2+m$, $m\in \Lambda_0$ are good. Then, Lemma \ref{bgs22} implies (after replacing $M_1^b$ by $M_1$ in the exponent and absorbing the factors):
$$
\|G_{\Lambda_0}(\theta_1,\theta_2,E)\|\le \lambda^{-1} e^{M_1}
$$
assuming that $\theta\in [0,1]^2\setminus\Theta$, where, for any unit line segment $L\subset \R^2$, we have
$$
|\Theta\cap L|_1\le e^{-M_1^{\rho}/2}.
$$
The argument behind applying Lemma \ref{bgs22} is exactly the same as the one that leads to \cite[equation (4.23)]{BGS}. Note that, in out notation, $\Theta$ is a 2D set.

Let $L\subset\R^2$ be a line segment. Introduce new orthonormal coordinates 
$$
\theta(\eta_1,\eta_2)=(\theta_1(\eta_1,\eta_2),\theta_2(\eta_1,\eta_2)).
$$
so that $L$ is parametrized by $(\eta_1,0)$ as $\eta_1\in [0,1]$, and consider Green's functions $G_{\Lambda}(\theta_1(\eta_1,\eta_2),\theta_2(\eta_1,\eta_2),E)$. Denote by $\Xi$ the set of $\eta_1$ such that $(\theta_1(\eta_1,0),\theta_2(\eta_1,0))$ runs over $\Theta\cap L$. Then $|\Xi|\le e^{-M_1^{\rho}/2}$. Without loss of generality, one can increase the length of $L$ by $2$, so that
$$
\|G_{\Lambda_0}(\theta_1(\eta_1,0),\theta_2(\eta_1,0),E)\|\le e^{M_1},\quad \forall \eta_1\in [-1,2]\setminus\Xi,
$$
with the same bound on $|\Xi|$.

Apply Cartan's lemma (Proposition \ref{prop_cartan}) with:
\begin{itemize}
	\item $A(z)=\frac{1}{\lambda}R_{\Lambda_0}(H(\theta_1(z+1/2,0),\theta_2(z+1/2,0))-E)R_{\Lambda_0}$.
	\item $[1,D]$ enumerates the points of $[-N,N]^2$.
	\item $\Lambda=\Lambda_0\setminus\Lambda_{\ast}$.
	\item $D_0=M_0^{C'_{\mathrm{int}}}C_{\mathrm{dio}}'(N^{3/4+3\delta_{\mathrm{dio}}})$.
	\item $[-\delta_1,\delta_1]=[-3/2,3/2]$, $\delta_2=\delta_2(v)$ is the width of the strip to which $v$ can be analytically extended.
	\item $B_1= C(v,m_{\mathrm{int}})$ is the norm bound for $A(z)$ in that strip.
	\item $B_2=2M_0^2 e^{(2 M_0)^b}$.
	\item $B_3= e^{M_1}$.
\end{itemize}
In order for the lemma to produce a meaningful bound, we would need to verify Assumption (3), that is, the estimate of the measure of $\Xi$; we need
$$
e^{-5 M_0}\approx e^{-M_1^{\rho}/2}<\frac{\delta_2(v)}{1000 (1+ C(v,m_{\mathrm{int}}))(1+2M_0^2 e^{(2 M_0)^b})}\approx C_1(v,m_{\mathrm{int}})M_0^{-2}e^{-(2M_0)^b}.
$$
Since $b<1$, this bound is clearly satisfied at a sufficiently large scale (depending only on $v$). In order to obtain the final estimate \eqref{slices}, we take $\varkappa=e^{-N^b}$. The condition preceding \eqref{eq_cartan} becomes
$$
e^{-N^b}<(1+B_1+B_2)^{-10 D_0}=\l\{1+C(v,m_{\mathrm{int}})+2M_0^2 e^{(2 M_0)^b})\r\}^{-M_0^{C'_{\mathrm{int}}}C_{\mathrm{dio}}' N^{3/4+3\delta_{\mathrm{dio}}}}
$$
or, after absorbing $1+C(v,m_{\mathrm{int}})$,
$$
N^b > M_0^{C'_{\mathrm{int}}}C_{\mathrm{dio}}' N^{3/4+3\delta_{\mathrm{dio}}}\l\{(2M_0)^b+\log(2M_0^2)\r\}.
$$
Recall that $N=[M_0^{C_1}]$. This implies the following condition on $C_1$ and $b$:
$$
C_1 b>C'_{\mathrm{int}}+C_1(3/4+3\delta_{\mathrm{dio}})+b+C_{\mathrm{dio}}'',
$$
which can be satisfied as long as $b>3/4+3\delta_{\mathrm{dio}}$ and $\overline{N}_0$ is sufficiently large, depending on $C_{\mathrm{dio}}$.

Finally, we check the conditions under which the right hand side of \eqref{eq_cartan} gives at least as good estimate as the one in the right hand side of \eqref{slices}: that is, when
$$
\frac{-c\log \varkappa}{D_0\log(D_0+B_1+B_2+B_3)}>N^{3\rho},
$$
or
$$
c N^b >N^{3\rho} D_0 \log(D_0+B_1+B_2+B_3)\approx N^{3\rho} D_0 M_1.
$$
Recalling the definition of $M_1$ and absorbing the constants, we arrive to the following sufficient condition (note that $B_3$ dominates $D_0$, $B_1$, $B_2$, assuming again that $\overline{N}_0$ is sufficiently large, depending on $C_{\mathrm{dio}}$):
$$
N^b>D_0 M_0^{1/\rho}N^{3\rho},
$$
which transforms into the following condition on $C_1$:
$$
C_1 b>C'_{\mathrm{int}}+C_1(3/4+3\delta_{\mathrm{dio}}+3\rho)+1/\rho+C''_{\mathrm{dio}}.
$$
The last condition can be satisfied if
$$
b-3/4-3\delta_{\mathrm{dio}}-3\rho>0,
$$
$$
C_1>\frac{C'_{\mathrm{int}}+C''_{\mathrm{dio}}+1/\rho}{b-3/4-3\delta_{\mathrm{dio}}-3\rho}.
$$
and then choosing $\overline{N}_0$ in order to satisfy earlier assumptions on a sufficiently large scale.\,\,\qed
\begin{remark}
The only place where a bound on $m_{\mathrm{int}}$ is used, is the Assumption (1) of Cartan's lemma.
\end{remark}
\subsection{Proofs of some semi-algebraic facts} We start from the following quantitative triangulation theorem due to Yomdin and Gromov formulated in \cite{Gr}, see also \cite{Y,Burget,Pila}. Let $\Delta_k$ be the standard $k$-dimensional simplex.
\begin{prop}
\label{prop_triangulation}
Fix $r\in \Z_+$. Any closed semi-algebraic set $Y\subset[0,1]^d$ can be triangulated into $(\deg Y+1)^{C(d,r)}$ simplices, where for every closed triangulating $k$-simplex $\Delta\subset Y$ there exists an algebraic homeomorphism 
$$
h_{\Delta}\colon \Delta_k\to \Delta,\quad \deg h_{\Delta}\le (\deg Y+1)^{C(n,r)},
$$
such that $\l.h_{\Delta}\r|_{\mathrm{int}\, \Delta_k}$ is real analytic with non-vanishing differential, and $\|D_r h_{\Delta}\|\le 1$.
\end{prop}

{\noindent \bf Proof of (sa5)}. Apply Proposition \ref{prop_triangulation} with $r= 2$. Since the interior point of any $d$-dimensional simplex is also an interior point of $S$, the boundary  $\partial S$ is contained in the union of all triangulating simplices of dimensions $\le d-1$. The bound on $\|D_r h_{\Delta}\|$ will imply the bound on the area of $\partial S$.
\vskip 2mm

{\noindent \bf Proof of (sa6)}. Suppose $|S|\ge \varepsilon_1$ and $S$ does not contain a ball of radius $\delta$. Then $S\subset \partial S+B_{\delta}$. Applying (sa4) with $\varepsilon=\delta$, we obtain an upper bound $|S|\le (\deg S)^{C(d)}\delta$. The contradiction obtained at $\delta<(\deg S)^{C(d)}\varepsilon_1$, proves the first claim.

To prove the second claim, apply Proposition \ref{prop_triangulation} to $\mathcal C$ and obtain a piece $\mathcal C_1\subset  \mathcal C$ with the following properties: $\mathcal C_1$ is a homeomorphic image of $[0,1]$ under an algebraic map $h$, which is real analytic with non-zero derivative on $(0,1)$. Moreover, the length of $\mathcal C_1$ satisfies $|\mathcal C_1|\ge B^{-C}L$. With a loss of extra $B^C$ (absorbed into the previous estimate), we can also assume that the direction of tangent vector to $\mathcal C_1$ does not change more than by, say, $\pi/100$. Then, one can algebraically re-parametrize $\mathcal C_1$ by the distance from the initial point $h(0)$ and obtain an interval $I_1$ of length $B^{-C} L$ and an algebraic function $h_1\colon I_1\to \mathcal C_1$ that parametrizes $\mathcal C_1$ with $\frac12 \le |h_1'|\le 2$. With these preparations, one can apply the first claim to $I_1\times I$: the splitting of $\mathcal C\times I$ into $K$ semialgebraic pieces induces a semialgebraic splitting of $I_1\times I$ into $K$ pieces. At least one of the pieces has area $\ge B^{-C} L^2 K^{-1}$, and hence contains a ball $B$ of radius $B^{-C} L^2 K^{-1}$. By picking a smaller square inside of that ball, one can find intervals $I_2,J$ such that $I_2\times J\subset B$ and $|I_2|,J|\ge B^{-C} L^2 K^{-1}$. To conclude, one can pick $\mathcal C'=h_1(I_2)$.
\vskip 2mm

{\noindent \bf Proof of (sa8)}. The part regarding algebraic local charts is standard, using the fact that the image of a semialgebraic set under an algebraic map is semialgebraic, with the control on the degrees.

To explain the proof for the provided example, consider the set 
$$
\mathcal B=\{(x,y)\in \R^{d+1}\times \R^{d+1}\colon |x-y|\ge \eta,\,\,[x,y]\subset \A\},
$$
where $[x,y]$ denotes the line segment connecting $x$ and $y$. Clearly, $\mathcal B$ is semialgebraic of degree $\le (\deg\A)^{C(d)}$. Define a function
$$
f\colon \mathcal B\to \mathbb S^d,\quad f(x,y)=\frac{y-x}{|y-x|}.
$$
Then $f$ is algebraic, and $\Xi=f(\mathcal B)$ is a semi-algebraic subset of $\mathbb S^d$ of degree $\le (\deg\A)^{C(d)}$.
\vskip 2mm

{\noindent \bf Proof of (sa9)}. Since $S$ is connected and admits a triangulation, it is also path connected. Let $p,q\in S$, and $f\colon [0,1]\to S$ be a continuous path: $f(0)=p$, $f(1)=q$. Apply Proposition \ref{prop_triangulation} to $S$. For each triangulating simplex $\Delta$, let $[a,b]$ be the smallest interval containing $f^{-1}(\Delta)$. Then, replace the part of the path $f$ between $a,b$ by $h_{\Delta}(L)$, where $L$ is the line segment in $\Delta$ connecting $h_{\Delta}^{-1}(f(a))$ and $h_{\Delta}^{-1}(f(b))$. By applying this operation (at most once for each simplex), one gets a path connecting $p$ and $q$, that visits each simplex only once, and is a straight line segment in the local coordinates on each simplex. Clearly, this new path satisfies all stated properties. See also \cite[Section 5.2]{ARAG}.

\section{Acknowledgements} We would like to thank Svetlana Jitomirskaya for drawing our attention to the problem, and also thank her, Thomas Spencer, and Alexander Volberg for valuable discussions. We are also very grateful to the anonymous referee, whose kind suggestions led to a major improvement in the quality of the text. 

The first author was supported by NSF DMS--1800640 grant ``New Decouplings and Applications''. The second author was supported by the NSF DMS--1600422/DMS--1758326 grant ``Spectral Theory of Periodic and Quasiperiodic Quantum Systems'', and earlier partially supported by NSF DMS--1401204. A significant portion of this work was completed during the second author's membership at Institute for Advanced Study in 2016/2017, and he would like to thank IAS for their hospitality. This material is also based upon work supported by the National Science Foundation under agreement No. DMS--1128155.  Any opinions, findings and conclusions or recommendations expressed in this material are those of the author(s) and do not necessarily reflect the views of the National Science Foundation.

\end{document}